\documentclass[a4paper,11pt,reqno,twoside]{amsart}
\usepackage{Kloosterman_depth}
\usepackage{float}
\floatstyle{boxed}
\newfloat{Fig}{htbp}{Fig}[section]
\floatname{Fig}{Fig}
\begin{document}%
\title[Kloosterman paths of prime powers moduli]{Kloosterman paths of prime powers moduli}
\author[G. Ricotta]{Guillaume Ricotta}
\address{Universit\'{e} de Bordeaux \\
Institut de Math\'{e}matiques de Bordeaux \\
351, cours de la Lib\'{e}ration \\
33405 Talence cedex \\
France}
\email{Guillaume.Ricotta@math.u-bordeaux.fr}
\author[E. Royer]{Emmanuel Royer}
\address{Laboratoire de Math\'{e}matiques \\
Campus universitaire des C\'{e}zeaux \\
3 place Vasarely \\
TSA 60026 \\
CS 60026 \\
63178 Aubi\`{e}re Cedex \\
France}
\email{emmanuel.royer@math.univ-bpclermont.fr}
\date{Version of \today} 
\subjclass{11T23, 11L05, 60F17, 60G17, 60G50.}
\keywords{Kloosterman sums, moments, random Fourier series, probability in Banach spaces.}
\begin{abstract}
In \cite{KoSa}, the authors proved, using a deep independence result of Kloosterman sheaves, that the polygonal paths joining the partial sums of the normalized classical Kloosterman sums $S\left(a,b_0;p\right)/p^{1/2}$ converge in the sense of finite distributions to a specific random Fourier series, as $a$ varies over $\left(\mathbb{Z}/p\mathbb{Z}\right)^\times$, $b_0$ is fixed in $\left(\mathbb{Z}/p\mathbb{Z}\right)^\times$ and $p$ tends to infinity among the odd prime numbers. This article considers the case of $S\left(a,b_0;p^n\right)/p^{n/2}$, as $a$ varies over $\left(\mathbb{Z}/p^n\mathbb{Z}\right)^\times$, $b_0$ is fixed in $\left(\mathbb{Z}/p^n\mathbb{Z}\right)^\times$, $p$ tends to infinity among the odd prime numbers and $n\geq 2$ is a fixed integer. A convergence in law in the Banach space of complex-valued continuous function on $[0,1]$ is also established, as $(a,b)$ varies over $\left(\mathbb{Z}/p^n\mathbb{Z}\right)^\times\times\left(\mathbb{Z}/p^n\mathbb{Z}\right)^\times$, $p$ tends to infinity among the odd prime numbers and $n\geq 2$ is a fixed integer. This is the analogue of the result obtained in \cite{KoSa} in the prime moduli case.
\end{abstract}
\maketitle
\begin{center}
\textit{In memory of Kevin Henriot.}
\end{center}
%
\section{Introduction and statement of the results}%
%
%
The shape of the path induced by various partial exponential sums has been considered by many people since the seventies. See for instance \cite{MR0429787}, \cite{MR737174} for the case of Gau\ss{} sums, \cite{MR817102} for polynomial exponential sums of higher degree, \cite{MR3081779}, \cite{BoGoGaSo} and \cite{MR2276774} for the case of character sums. Very recently, E.~Kowalski and W.~Sawin successfully investigated the case of partial Kloosterman sums of prime moduli in \cite{KoSa}. The main purpose of this work is to consider the case of partial Kloosterman sums to prime power moduli and to give a probabilistic meaning to graphs like the one given in Figure \ref{fig_K2(p=61,n=2)}\footnote{The axes are orthonormal but a rotation by $\pi/2$ has been applied to the real plot of $t\mapsto\mathsf{Kl}_{67^2}(t;(1,1))$.}.
\begin{Fig}
\begin{center}
\includegraphics[scale=0.6,angle=90]{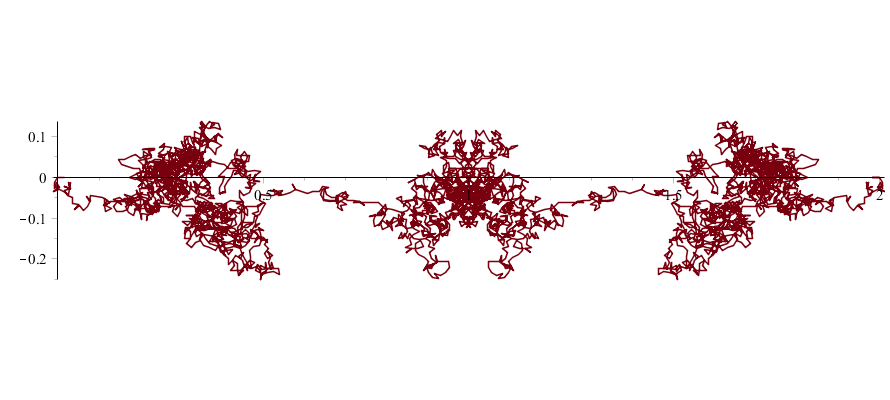}
\end{center}
\caption{Plot of $t\mapsto\mathsf{Kl}_{67^2}(t;(1,1))$}
\label{fig_K2(p=61,n=2)}
\end{Fig}
\par
More precisely, let $p$ be a prime number and $n\geq 1$ an integer. For $a$ and $b$ in $\mathbb{Z}/p^n\mathbb{Z}$, the corresponding normalized Kloosterman sum of modulus $p^n$ is the real number given by
\begin{equation*}
\mathsf{Kl}_{p^n}(a,b)\coloneqq\frac{1}{p^{n/2}}S\left(a,b;p^n\right)=\frac{1}{p^{n/2}}\sum_{\substack{1\leq x\leq p^n \\
p\nmid x}}\
e\left(\frac{ax+b\overline{x}}{p^n}\right)
\end{equation*}
where as usual $\overline{x}$ stands for the inverse of $x$ modulo $p^n$ and $e(z)\coloneqq\exp{(2i\pi z)}$ for any complex number $z$. For $a$ and $b$ in $\left(\mathbb{Z}/p^n\mathbb{Z}\right)^\times$, the associated partial sums are the $\varphi(p^n)=p^{n-1}(p-1)$ complex numbers
\begin{equation*}
\mathsf{Kl}_{j;p^n}(a,b)\coloneqq\frac{1}{p^{n/2}}\sum_{\substack{1\leq x\leq j \\
p\nmid x}}e\left(\frac{ax+b\overline{x}}{p^n}\right)
\end{equation*}
for $j$ in $J_p^n\coloneqq\left\{j\in\{1,\dots,p^n\}, p\nmid j\right\}$. If we write \(J_p^n=\{j_1,\dotsc,j_{\varphi(p^n)}\}\) with
\begin{equation*}
j_1<j_2<\dots<j_{\varphi(p^n)}
\end{equation*}
then the corresponding Kloosterman path $\gamma_{p^n}(a,b)$ is defined by %
\[%
\gamma_{p^n}(a,b)=\bigcup_{j=1}^{\varphi(p^n)-1}\left[\mathsf{Kl}_{j_i;p^n}(a,b),\mathsf{Kl}_{j_{i+1};p^n}(a,b)\right]. %
\]
This is %
the polygonal path obtained by concatenating the closed segments
\begin{equation*}
\left[\mathsf{Kl}_{j_1;p^n}(a,b),\mathsf{Kl}_{j_2;p^n}(a,b)\right]
\end{equation*}
for $j_1$ and $j_2$ two consecutive indices in $J_p^n$. Finally, one defines a continuous map on the interval $[0,1]$
\begin{equation*}
t\mapsto\mathsf{Kl}_{p^n}(t;(a,b))
\end{equation*}
by parametrizing the path $\gamma_{p^n}(a,b)$, each segment $\left[\mathsf{Kl}_{j_1;p^n}(a,b),\mathsf{Kl}_{j_2;p^n}(a,b)\right]$ for $j_1$ and $j_2$ two consecutive indices in $J_p^n$ being parametrized linearly by an interval of length $1/(\varphi(p^n)-1)$.
\par
For a fixed $b_0$ in $\left(\Z/p^n\Z\right)^\times$, the function $a\mapsto\mathsf{Kl}_{p^n}(\ast;(a,b_0))$ is viewed as a random variable on the probability space $\left(\mathbb{Z}/p^n\mathbb{Z}\right)^\times$ endowed with the uniform probability measure with values in the Banach space of complex-valued continuous functions on $[0,1]$ endowed with the supremum norm, say $C^0([0,1],\C)$.
\begin{remark}
In particular, with our definition, $\mathsf{Kl}_{p^n}(0;(a,b))$ is defined y
\begin{equation*}
\exists\lim_{t\to 0}\mathsf{Kl}_{p^n}(t;(a,b))=\frac{1}{p^{n/2}}e\left(\frac{a+b}{p^n}\right)=\mathsf{Kl}_{p^n}(0;(a,b)).
\end{equation*}
The Kloosterman path does not start at the origin, in contrast with \cite{KoSa}.
\end{remark}
Let $\mu$ be the probability measure given by
\begin{equation}\label{eq_def_mu_1}
\mu=\frac{1}{2}\delta_0+\mu_1
\end{equation}
for the Dirac measure $\delta_0$ at $0$ and
\begin{equation*}
\mu_1(f)=\frac{1}{2\pi}\int_{x=-2}^{2}\frac{f(x)\mathrm{d}x}{\sqrt{4-x^2}}
\end{equation*}
for any real-valued continuous function $f$ on $[-2,2]$.
\begin{theoint}[Convergence of finite distributions]\label{theo_A}
Let $n\geq 2$ be a fixed integer. 
For any odd prime number \(p\), fix an element \(b_0\) in $\left(\Z/p^n\Z\right)^\times$. %
Let $\left(U_h\right)_{h\in\Z}$ be a sequence of independent identically distributed random variables of probability law $\mu$ defined in \eqref{eq_def_mu_1} and let $\mathsf{Kl}$ be the $C^0([0,1],\C)$-valued random variable defined by
\begin{equation*}
\forall t\in[0,1],\quad \mathsf{Kl}(t)=tU_0+\sum_{h\in\Z^\ast}\frac{e(ht)-1}{2i\pi h}U_h.
\end{equation*}
The sequence of $C^0([0,1],\C)$-valued random variables $\mathsf{Kl}_{p^n}(\ast;(\ast,b_0))$ on $\left(\Z/p^n\Z\right)^\times$ converges in the sense of finite distributions\footnote{See appendix \ref{sec_proba} for a precise definition of the convergence in the sense of finite distributions.} to the $C^0([0,1],\C)$-valued random variable $\mathsf{Kl}$ as $p$ tends to infinity among the prime numbers.
\end{theoint}
\begin{remark}
We have chosen to parametrize the partial sums of the Kloosterman sums so that successive sums always correspond to adding one more term. This implies that partial sums at integers divisible by $p$ are not defined. Another definition would be to define $\mathsf{Kl}_{j;p^n}(a,b)$ for all integer $j$ and to interpolate in the usual way. The geometric path, namely the image of $t\mapsto\mathsf{Kl}_{p^n}(t;(a,b))$, would be unchanged and there is no doubt that the same results hold for this different definition.
\end{remark}
\begin{remark}
All the main properties of the random variable $\mathsf{Kl}$ are given in Proposition \ref{propo_kahane}. As already said, this theorem is the analogue of the result proved by E.~Kowalski and W.~Sawin in \cite{KoSa} when $n=1$ for a different random Fourier series given by
\begin{equation*}
\forall t\in[0,1],\quad K(t)\coloneqq t\mathsf{ST}_0+\sum_{h\in\Z^\ast}\frac{e(ht)-1}{2i\pi h}\mathsf{ST}_h
\end{equation*}
where $\left(\mathsf{ST}_h\right)_{h\in\Z}$ is an independent identically distributed sequence of random variables of probability law $\mu_{ST}$, the classical Sato-Tate measure also called the semi-circle law. The fact that $K$ and $\mathsf{Kl}$ have the same analytic shape heavily depends on the completion method. The fact that $K$ and $\mathsf{Kl}$ are different on a probabilistic point of view is not very surprising since $\mathsf{Kl}_{p^n}(a,b)$ is a sum over a finite field when $n=1$, which requires deep techniques from algebraic geometry, and a character sum when $n\geq 2$, which can be computed explicitly via elementary but not so easy techniques. Thus, the fact that Kloosterman paths of prime moduli and of prime powers moduli behave differently on a probabilistic point of view is quite expected.
\end{remark}
\begin{remark}
Nevertheless, the referee kindly informed us that both this measure $\mu$ and the random series $\mathsf{Kl}$ occur when dealing with the path induced by Sali\'{e} sums of prime moduli. In addition, let us recall that $\mu_{ST}$ is the direct image under the trace map of the probability Haar measure on the compact group $SU_2(\C)$ whereas, according to \cite[Remark 1.2]{MR2646758}, $\mu$ is the direct image under the trace map of the probability Haar measure on the normalizer of a maximal torus in $SU_2(\C)$.
\end{remark}
\begin{remark}\label{remark_111}
In particular, choosing $t=1$, Theorem \ref{theo_A} implies that the normalized Kloosterman sums $\mathsf{Kl}_{p^n}(a,b_0)$ get equidistributed in $[-2,2]$ with respect to the measure $\mu$, as $a$ ranges over $\left(\Z/p^n\Z\right)^\times$ and $p$ tends to infinity among the odd prime numbers for a fixed integer $n\geq 2$ and $b_0$ is a fixed element in $\left(\Z/p^n\Z\right)^\times$.
\end{remark}
\begin{remark}
It is worth mentioning that the proof of this theorem requires A.~Weil's version of the Riemann hypothesis in one variable. See Proposition \ref{propo_cardinality}.
\end{remark}
The function $(a,b)\mapsto\mathsf{Kl}_{p^n}(t;(a,b))$ is viewed as a $C^0([0,1],\C)$-valued random variable on the probability space $\left(\mathbb{Z}/p^n\mathbb{Z}\right)^\times\times\left(\mathbb{Z}/p^n\mathbb{Z}\right)^\times$ endowed with the uniform probability measure. Theorem \ref{theo_A} trivially implies that the sequence of $C^0([0,1],\C)$-valued random variables $\mathsf{Kl}_{p^n}(\ast;(\ast,\ast))$ converges in the sense of finite distributions to the $C^0([0,1],\C)$-valued random variable $\mathsf{Kl}$ as $p$ tends to infinity among the prime numbers too.
\begin{theoint}[Convergence in law]\label{theo_B}
Let $n\geq 2$ be a fixed integer and $p$ be an odd prime number. The sequence of $C^0([0,1],\C)$-valued random variables $\mathsf{Kl}_{p^n}(\ast;(\ast,\ast))$ on $\left(\Z/p^n\Z\right)^\times\times\left(\Z/p^n\Z\right)^\times$ converges in law\footnote{See appendix \ref{sec_proba} for a precise definition of the convergence in law in the Banach space $C^0([0,1],\C)$.} to the $C^0([0,1],\C)$-valued random variable $\mathsf{Kl}$ as $p$ tends to infinity among the prime numbers.
\end{theoint}
\begin{remark}
Once again, this theorem is the analogue of the result proved by E.~Kowalski and W.~Sawin in \cite{KoSa} when $n=1$.
\end{remark}
\begin{remark}
For a fixed $n\geq 2$ and a fixed $b_0$ in $\left(\Z/p^n\Z\right)^\times$, we expect that the sequence of $C^0([0,1],\C)$-valued random variables $\mathsf{Kl}_{p^n}(\ast;(\ast,b_0))$ on $\left(\Z/p^n\Z\right)^\times$ converges in law to the $C^0([0,1],\C)$-valued random variable $\mathsf{Kl}$ as $p$ tends to infinity among the prime numbers too. Nevertheless, such result seems to be out of reach given the current technology. It relies on expected uniform non-trivial individual bounds for incomplete Kloosterman sums
\begin{equation*}
\frac{1}{p^{n/2}}\sum_{x\in I}e\left(\frac{ax+b_0\overline{x}}{p^n}\right)\ll p^{-\delta}
\end{equation*}
for some $\delta>0$ and where $I$ is an interval of $\left(\Z/p^n\Z\right)^\times$ of length close to $p^{n/2}$. See \cite[Remark 3.3]{KoSa} and \cite[Page 52]{Ko} for a discussion on such issues in the prime moduli case.
\end{remark}
\par
In \cite{KoSa}, the authors deduce from their limit theorems the distribution of the maximum of the partial sums of prime moduli they consider. Their techniques would lead to a straightforward analogue in the case of prime powers moduli investigated in this work.
\par
One can mention that it seems quite natural to consider the same questions in the regime\footnote{Or even worse any intermediate regime.} $p$ a fixed prime number and $n\geq 2$ tends to infinity. This problem, both theoretically and numerically, seems to be of completely different nature. 
\par
Finally, it makes sense to consider the distribution of paths associated to other exponential sums of prime powers moduli and to ask whether a distribution result remains true. For instance, one could be tempted to look at
\begin{equation*}
\mathsf{K}_{p^n}(a)=\frac{1}{p^{n/2}}\sum_{1\leq x\leq p^n}^{\ast}e\left(\frac{f_a(x)}{p^n}\right)
\end{equation*}
where $f_a=g_a/h_a$ with $g_a$ and $h_a$ in $\mathbb{Z}[x]$ depending on a parameter $a$ modulo $p^n$. The symbol $\ast$ means that the summation is over the elements $x$ satisfying $p\nmid h_a(x)$. These exponential sums can be computed explicitly. See \cite[Lemma 12.2, Lemma 12.13]{IwKo} for instance. One key step would be to evaluate asymptotically  
\begin{equation*}
\frac{1}{\varphi\left(p^n\right)}\sum_{a\in\left(\mathbb{Z}/p^n\mathbb{Z}\right)^\times}\prod_{\tau\in\mathbb{Z}/p^n\mathbb{Z}}\mathsf{K}_{p^n}(a+\tau)^{\mu(\tau)}
\end{equation*}
for $\uple{\mu}=\left(\mu(\tau)\right)_{\tau\in\mathbb{Z}/p^n\mathbb{Z}}$ a $p^n$-tuple of non-negative integers. 
\par
\noindent{\textbf{Organization of the paper. }}%
The explicit description of the Kloosterman paths is given in Section \ref{sec_path}. The relevant random Fourier series, which occurs as an asymptotic process in Theorem \ref{theo_A} and Theorem \ref{theo_B}, is defined and studied in Section \ref{sec_process}. Section \ref{sec_moments} contains the asymptotic evaluation of the moments of the random variable $\mathsf{Kl}_{p^n}(\ast;(\ast,\ast))$ whereas the tightness of this sequence of random variables is established in Section \ref{sec_tight}. The proofs of Theorem \ref{theo_A} and Theorem \ref{theo_B} are completed in Section \ref{sec_final}. A probabilistic toolbox is provided in appendix \ref{sec_proba}.
\begin{notations}
The main parameter in this paper is an odd prime $p$, which tends to infinity. Thus, if $f$ and $g$ are some $\C$-valued function of the real variable then the notations $f(p)=O_A(g(p))$ or $f(p)\ll_A g(p)$ mean that $\abs{f(p)}$ is smaller than a "constant", which only depends on $A$, times $g(p)$ at least for $p$ large enough.
\par
$n\geq 2$ is a fixed integer.
\par
For any real number $x$ and integer $k$, $e_k(x)\coloneqq\exp{\left(\frac{2i\pi x}{k}\right)}$.
\par
For any finite set $S$, $\vert S\vert$ stands for its cardinality. 
\par
We will denote by $\epsilon$ an absolute positive constant whose definition may change from one line to the next one.
\par
The notation $\sum^{\times}$ means that the summation is over a set of integers coprime with $p$.
\par
Finally, if $\mathcal{P}$ is a property then $\delta_{\mathcal{P}}$ is the Kronecker symbol, namely $1$ if $\mathcal{P}$ is satisfied and $0$ otherwise.
\end{notations}
\begin{merci}%
\par
The authors would like to thank the referee for her or his unusual careful reading of the manuscript and very useful suggestions that improved the presentation of the paper.
\par
The authors would like to thank E.~Kowalski for his encouragements and for sharing with us his enlightening lectures notes \cite{Ko}. They also thank F.~Martin for fruitful discussion related to Proposition \ref{propo_tricky_counting}.
\par
Part of this paper was worked out in Universit\'{e} Blaise Pascal (Clermont-Ferrand, France) in June 2016. The first author would like to thank this institution for its hospitality and inspiring working conditions. 
\end{merci}
\section{Explicit description of the Kloosterman path}\label{sec_path}%

Let us construct the Kloosterman path $\gamma(a,b)$ for $a$ and $b$ in $\left(\mathbb{Z}/p^n\mathbb{Z}\right)^\times$. %

We enumerate the partial Kloosterman sums and define \(z_j((a,b);p^n)\) to be the \(j\)-th term of \(\left(\mathsf{Kl}_{j;p^n}(a,b)\right)_{j\in J_p^n}\). More explicitly, we organise the partial Kloosterman sums in \(p^{n-1}\) blocks each of them containing \(p-1\) successive sums. For $1\leq k\leq p^{n-1}$, the \(k\)-th block contains \(\mathsf{Kl}_{(k-1)p+1;p^n}(a,b),\dotsc,\mathsf{Kl}_{kp-1;p^n}(a,b)\). These sums are numbered by defining %
\[%
z_{(k-1)(p-1)+\ell}((a,b);p^n)=\mathsf{Kl}_{(k-1)p+\ell;p^n}(a,b)\qquad (1\leq \ell\leq p-1). %
\]
It implies that the enumeration is given by %
\begin{equation}\label{eq_etoile}%
z_{j}((a,b);p^n)=\mathsf{Kl}_{j+\left\lfloor\frac{j-1}{p-1}\right\rfloor;p^n}(a,b)\qquad (1\leq j<\varphi(p^n))
\end{equation}

For any \(j\in\{1,\dotsc,\varphi(p^n)-1\}\), we parametrize the segment \[\left]z_{j}((a,b);p^n),z_{j+1}((a,b);p^n)\right]\] and obtain the parametrization of \(\gamma_{p^n}(a,b)\) given by %
\[%
\forall t\in[0,1],\, \mathsf{Kl}_{p^n}(t;(a,b))=\alpha_j((a,b);p^n)\left(t-\frac{j-1}{\varphi(p^n)-1}\right)+z_{j}((a,b);p^n) %
\]
with %
\[%
\alpha_j((a,b);p^n)=(\varphi(p^n)-1)\left(z_{j+1}((a,b);p^n)-z_{j}((a,b);p^n)\right) %
\]
and 
\[%
j=\left\lceil\left(\varphi(p^n)-1\right)t\right\rceil. %
\]
Since \(\left]z_{j}((a,b);p^n),z_{j+1}((a,b);p^n)\right]\) has length \(p^{-n/2}\), we have %
\begin{equation}\label{eq_tec_0}
\abs{\alpha_j((a,b);p^n)}\leq\frac{\varphi(p^n)-1}{p^{n/2}}
\end{equation}
and %
\begin{equation}\label{eq_tec}
\left\vert\mathsf{Kl}_{p^n}(t;(a,b))-z_{j}((a,b);p^n)\right\vert\leq\frac{1}{p^{n/2}}.
\end{equation}
\section{On the relevant random Fourier series}\label{sec_process}%
The moments of the measure $\mu$ defined in \eqref{eq_def_mu_1} are given by
\begin{equation}\label{eq_mu_1}
\int_{x\in\R}x^m\mathrm{d}\mu(x)=\begin{cases}
1 & \text{if $m=0$,} \\
\frac{\delta_{2\mid m}}{2}\binom{m}{m/2} & \text{otherwise.}
\end{cases}
\end{equation}
\par
Let $U$ be a random variable of law $\mu$ on a probability space $\left(\Omega,\mathcal{A},P\right)$. By \eqref{eq_mu_1}, the value of the expectation of such random variable is $0$ and its variance equals $1$. In addition, $\mu$ is also the law of the random variable $-U$ since the probability measure $\mu$ is symmetric.
\par
Let $\left(U_h\right)_{h\in\Z}$ be a sequence of independent random variables of law $\mu$ on a probability space $\left(\Omega,\mathcal{A},P\right)$. One defines for $t$ in $[0,1]$ the symmetric partial sums
\begin{equation*}
\mathsf{Kl}_H(t;\omega)\coloneqq tU_0(\omega)+\sum_{1\leq\abs{h}\leq H}\frac{e(ht)-1}{2i\pi h}U_h(\omega)
\end{equation*}
for any integer $H\geq 1$ and any $\omega\in\Omega$. %
Let \(t\in[0,1]\) and $\omega\in\Omega$. If \(\mathsf{Kl}_H(t;\omega)\) has a limit when \(H\) tends to infinity, we denote by \(\mathsf{Kl}(t;\omega)\) this limit, namely  %
\begin{equation*}
\mathsf{Kl}(t;\omega)\coloneqq tU_0(\omega)+\sum_{h\in\Z^\ast}\frac{e(ht)-1}{2i\pi h}U_h(\omega)
\end{equation*}
It turns out that $\mathsf{Kl}(t;\omega)$ is closely related to the set of Fourier random series, which have been intensively studied in \cite{MR0254888}.
\begin{proposition}[Properties of the random series]\label{propo_kahane}
The following properties hold.
\begin{itemize}
\item
For any $t$ in $[0,1]$, the random series $\mathsf{Kl}(t;\ast)$ converges almost surely, hence in law.
\item
For almost all $\omega\in\Omega$, the random series $\mathsf{Kl}(\ast;\omega)$ is a continuous function on $[0,1]$.
\item
For any $t$ in $[0,1]$, the Laplace transform
\begin{equation*}
\mathbb{E}\left(e^{\lambda\Re{\left(\mathsf{Kl}(t;\ast)\right)}+\mu\Im{\left(\mathsf{Kl}(t;\ast)\right)}}\right)
\end{equation*}
is well-defined for all non-negative integers $\lambda$ and $\mu$. In particular, $\mathsf{Kl}(\ast;\omega)$ has moments of all orders.
\item
Finally, for any $t$ in $[0,1]$,
\begin{equation}\label{eq_supnorm}
\abs{\abs{\mathsf{Kl}_H(t;\ast)}}_\infty\ll\log{(H)}
\end{equation}
and
\begin{equation}\label{eq_L1}
\left\vert\mathbb{E}\left(\left\vert\mathsf{Kl}(t;\ast)-\mathsf{Kl}_H(t;\ast)\right\vert\right)\right\vert\ll H^{-1/2}
\end{equation}
for any $H\geq 1$.
\end{itemize}
\end{proposition}
\begin{remark}
In particular, the map
\begin{equation*}
\begin{array}{cccl}
\mathsf{Kl}: & \left(\Omega,\mathcal{A},P\right) & \to & \left(C^0\left([0,1],\C\right),\abs{\abs{.}}_\infty\right) \\
& \omega & \mapsto & \begin{array}{cccl}
\mathsf{Kl}(\ast;\omega): & [0,1] & \to & \C \\
& t & \mapsto & \mathsf{Kl}(t;\omega)
\end{array}
\end{array}
\end{equation*}
defines a random variable on the probability space $\left(\Omega,\mathcal{A},P\right)$ with values in the Banach space of continuous complex-valued functions on the segment $[0,1]$ endowed with the supremum norm $\abs{\abs{.}}_\infty$.
\end{remark}
\begin{remark}
The proof is omitted since it is very close to the proof of \cite[Proposition 2.1]{KoSa}. The reader may have a look at \cite[Section 4]{Ko} too.
\end{remark}
\section{Asymptotics of complex moments}\label{sec_moments}%
In this section, $b_0$ is a \emph{fixed} element in $\left(\Z/p^n\Z\right)^\times$. Let $k\geq 1$ be a fixed integer, $\uple{t}=(t_1,\dots,t_k)$ be a fixed $k$-tuple of elements in $[0,1]$ with $t_1<\dots<t_k$, $\uple{n}=(n_1,\dots,n_k)$ and $\uple{m}=(m_1,\dots,m_k)$ be two fixed $k$-tuples of non-negative integers. Let us define
\begin{equation*}
\ell(\uple{m+n})\coloneqq\sum_{i=1}^k\left(m_i+n_i\right).
\end{equation*}
\par
The purpose of this section is to find an asymptotic formula for the complex moments defined by
\begin{equation}\label{eq_moments_bis}
\mathsf{M}_{p^n}(\uple{t};\uple{m},\uple{n};b_0)\coloneqq\frac{1}{\varphi(p^n)}\sum_{a\in\left(\mathbb{Z}/p^n\mathbb{Z}\right)^\times}\prod_{i=1}^k\overline{\mathsf{Kl}_{p^n}(t_i;(a,b_0))}^{m_i}\mathsf{Kl}_{p^n}(t_i;(a,b_0))^{n_i}.
\end{equation}
The following proposition describes the asymptotic expansion of these moments. Its proof will be given at the very end of this section since it requires a series of intermediate results.
\begin{proposition}[Asymptotic expansion of the moments]\label{propo_moment_tilde}
If 
\begin{equation}
p>\max{\left(\ell(\uple{m}+\uple{n}),2n-5\right)}
\end{equation}
then
\begin{multline*}
\mathsf{M}_{p^n}(\uple{t};\uple{m},\uple{n},b_0)=\mathbb{E}\left(\prod_{i=1}^{k}\overline{\mathsf{Kl}(t_i;\ast)}^{m_i}\mathsf{Kl}(t_i;\ast)^{n_i}\right) \\
+O_{\ell(\uple{m}+\uple{n}),\epsilon}\left(\log^{\ell(\uple{m}+\uple{n})}\left(p^n\right)\left(p^{-\frac{4(n-1)}{2^n}+\epsilon}+p^{-1/2}\right)\right)
\end{multline*}
for any $\epsilon>0$ and where the implied constant only depends on $\ell(\uple{m}+\uple{n})$ and $\epsilon$.
\end{proposition}
\par
For $a$ in $\left(\mathbb{Z}/p^n\mathbb{Z}\right)^\times$, let us define a step function on the segment $[0,1]$ by, for any $k\in\left\{1,\dots,p^{n-1}\right\}$,
\begin{equation}\label{eq_Kltilde}
\forall t\in\left(\frac{k-1}{p^{n-1}},\frac{k}{p^{n-1}}\right],\quad\widetilde{\mathsf{Kl}_{p^n}}(t;(a,b_0))\coloneqq\frac{1}{p^{n/2}}\sum_{1\leq x\leq x_k(t)}^{\mathstrut\quad\times}e_{p^n}\left(ax+b_0\overline{x}\right).
\end{equation}
where
\begin{equation*}
x_k(t)\coloneqq\varphi(p^n)t+k-1.
\end{equation*}
In addition, let us define for $h$ in $\mathbb{Z}/p^n\mathbb{Z}$ and $1\leq k\leq p^{n-1}$,
\begin{equation}\label{eq_Fourier}
\forall t\in\left(\frac{k-1}{p^{n-1}},\frac{k}{p^{n-1}}\right],\quad\alpha_{p^n}(h;t)\coloneqq\frac{1}{p^{n/2}}\sum_{1\leq x\leq x_k(t)}e_{p^n}\left(hx\right).
\end{equation}
These coefficients are nothing else than the discrete Fourier coefficients of the finite union of intervals given by $1\leq x\leq x_k(t)$ with $(p,x)=1$ for $1\leq k\leq p^{n-1}$. All their useful properties are encapsulated in the following lemma.
\begin{lemma}[The completion method]\label{lemma_step_0}
\strut\par
\begin{itemize}
\item
For $H_{p^n}$ any complete system of residues modulo $p^n$,
\begin{equation}\label{eq_KL_tilde}
\widetilde{\mathsf{Kl}_{p^n}}(t;(a,b_0))=\frac{1}{p^{n/2}}\sum_{h\in H_{p^n}}\alpha_{p^n}(h;t)\mathsf{Kl}_{p^n}(a-h,b_0).
\end{equation}
\item
For any integer $h$ and any real number $t\in[0,1]$,
\begin{equation}\label{eq_bound_alpha}
\alpha_{p^n}(h;t)\leq p^{n/2}\times\begin{cases}
1 & \text{if $\;h=0$,} \\
\frac{1}{2\abs{h}} & \text{\;if $\abs{h}\leq (p^n-1)/2$ and $h\neq 0$.}
\end{cases}
\end{equation}
\item
For any integer $h$ and any real number $t\in[0,1]$,
\begin{equation}\label{eq_approx_alpha}
\frac{1}{p^{n/2}}\alpha_{p^n}(h;t)=\beta(h;t)+O\left(\frac{1}{p^n}\right)
\end{equation}
where
\begin{equation*}
\beta(h;t)=\begin{cases}
t & \text{if $\;h=0$,} \\
\frac{e(ht)-1}{2i\pi h} & \text{otherwise.}
\end{cases}
\end{equation*}
\end{itemize}
\end{lemma}
\begin{remark}
The proof is omitted since it is very close to the proof of \cite[Lemma 2.3, Proposition 2.4]{KoSa}. The reader may have a look at \cite[Section 4]{Ko} too.
\end{remark}
\par
Let us also define the corresponding moment 
\begin{equation}\label{eq_moments_bis_tilde}
\widetilde{\mathsf{M}_{p^n}}(\uple{t};\uple{m},\uple{n};b_0)\coloneqq\frac{1}{\varphi(p^n)}\sum_{a\in\left(\mathbb{Z}/p^n\mathbb{Z}\right)^\times}\prod_{i=1}^k\overline{\widetilde{\mathsf{Kl}_{p^n}}(t_i;(a,b_0))}^{m_i}\widetilde{\mathsf{Kl}_{p^n}}(t_i;(a,b_0))^{n_i}.
\end{equation}
\par
The following lemma reveals that it is enough to prove an asymptotic formula for $\widetilde{\mathsf{M}_{p^n}}(\uple{t};\uple{m},\uple{n};b_0)$.
\begin{lemma}[Approximation of the moments]\label{lemma_step_1}
One has
\begin{equation*}
\mathsf{M}_{p^n}(\uple{t};\uple{m},\uple{n};b_0)=\widetilde{\mathsf{M}_{p^n}}(\uple{t};\uple{m},\uple{n};b_0)+O\left(\frac{\log^{\ell(\uple{m}+\uple{n})}(p^n)}{p^{n/2}}\right).
\end{equation*}
\end{lemma}
\begin{remark}
The proof is omitted but relies on Lemma \ref{lemma_step_0}, which implies that
\begin{equation}\label{eq_correctif}
\sum_{h\in H_{p^n}}\abs{\alpha_{p^n}(h;t)}\leq 4p^{n/2}\log{(p^n)}
\end{equation}
for $H_{p^n}=\left\{(1-p^n)/2,\dots,(p^{n}-1)/2\right\}$, which is admissible since $p$ is odd, and is close to the proof of \cite[Proposition 2.4]{KoSa}. The reader may have a look at \cite[Section 4]{Ko} too. Note that both Lemma \ref{eq_approx_alpha} and \eqref{eq_correctif} entail that
\begin{equation}\label{eq_bemol}
\left\vert\mathsf{Kl}_{p^n}(t;(a,b))-\widetilde{\mathsf{Kl}_{p^n}}(t;(a,b))\right\vert\leq\frac{6}{p^{n/2}}
\end{equation}
for any $a$, $b$ in $\left(\mathbb{Z}/p^n\mathbb{Z}\right)^\times$ and any $t\in[0,1]$.
\end{remark}
The crucial ingredient in the proof of Proposition \ref{propo_moment_tilde} is the asymptotic evaluation of the complete sums of products of shifted Kloosterman sums $\mathsf{S}_{p^n}(\uple{\mu};b_0)$ defined by
\begin{equation}\label{eq_Spmu}
\mathsf{S}_{p^n}(\uple{\mu};b_0)\coloneqq\frac{1}{\varphi(p^n)}\sum_{a\in\left(\mathbb{Z}/p^n\mathbb{Z}\right)^\times}\prod_{\tau\in \mathbb{Z}/p^n\mathbb{Z}}\mathsf{Kl}_{p^n}(a+\tau,b_0)^{\mu(\tau)}
\end{equation}
for $\uple{\mu}=\left(\mu(\tau)\right)_{\tau\in\mathbb{Z}/p^n\mathbb{Z}}$ a sequence of $p^n$-tuples of non-negative integers different from the $0$-tuple.
\par
The following notations will be used throughout this section.  \label{page_notations} Let us define for such sequence $\uple{\mu}$
\begin{eqnarray*}
\mathsf{T}(\uple{\mu}) & \coloneqq & \left\{\tau\in\Z/p^n\Z, \mu(\tau)\geq 1\right\}\subset\Z/p^n\Z, \\
\overline{\mathsf{T}}(\uple{\mu}) & \coloneqq & \left\{\tau\bmod{p}, \tau\in\mathsf{T}(\uple{\mu})\right\}\subset\Z/p\Z
\end{eqnarray*}
\par
Let $\mathsf{B}_{p^n}(\uple{\mu})$ be the subset of the $\abs{\mathsf{T}(\uple{\mu})}$-tuples $\uple{b}=\left(b_\tau\right)_{\tau\in\mathsf{T}(\uple{\mu})}$ of integers in $\{1,\dots,(p-1)/2\}$ satisfying
\begin{equation}\label{eq_asump_1}
\forall(\tau,\tau')\in\mathsf{T}(\uple{\mu})^2,\quad b_\tau^2-\tau\equiv b_{\tau'}^2-\tau'\bmod{p}
\end{equation}
and
\begin{equation}\label{eq_deuxieme}
\forall\tau\in\mathsf{T}(\uple{\mu}),\quad p\nmid b_\tau^2-\tau.
\end{equation}
\par
Let $\uple{\ell}=\left(\ell_\tau\right)_{\tau\in\mathsf{T}(\uple{\mu})}$ be a $\left\vert\mathsf{T}(\uple{\mu})\right\vert$-tuple of integers. For any integer $j$ in $\{1,\dots,n-1\}$, let us define
\begin{equation}\label{eq_mmm}
m_{\uple{b},\uple{\ell}}(j,j)=\sum_{\tau\in\mathsf{T}(\uple{\mu})}\ell_\tau\overline{b_\tau}^{2j-1}
\end{equation}
and the the following associated object
\begin{equation}\label{eq_tricky_counting}
\mathsf{N}(\uple{\mu},\uple{\ell};w)\coloneqq\sum_{\substack{\uple{b}\in\mathsf{B}_{p^n}(\uple{\mu}) \\
m_{\uple{b},\uple{\ell}}(1,1)\equiv w\bmod{p} \\
\forall j\in\{2,\dots,n-1\},\quad m_{\uple{b},\uple{\ell}}(j,j)\equiv 0\bmod{p}}}1
\end{equation}
for any $w$ modulo $p$.
\par
Finally, let
\begin{equation}\label{eq_apmu}
\mathsf{A}_{p^n}(\uple{\mu})\coloneqq\left\{a\in\left(\Z/p^n\Z\right)^\times, \forall\tau\in\mathsf{T}(\uple{\mu}), a+\tau\in\left(\left(\Z/pçn\Z\right)^\times\right)^2\right\}.
\end{equation}
Firstly, let us prove and recall some useful facts related to Kloosterman sums of prime powers moduli.
\begin{lemma}[Kloosterman sums of prime powers moduli]\label{lemma_kppm}
Let $p$ be an odd prime number satisfying $p\geq 2n-5$ and $a$ be an integer.
\begin{itemize}
\item
If $a$ is divisible by $p$ or $a$ is not a square modulo $p$ then $\mathsf{Kl}_{p^n}(a,1)=0$.
\item
If $a$ is a non-zero square modulo $p$ then
\begin{equation*}
\mathsf{Kl}_{p^n}(a,1)=2\left(\frac{s}{p^n}\right)\cos{\left(\frac{4\pi s}{p^n}+\theta_{p^n}\right)}
\end{equation*}
where
\begin{equation*}
\theta_{p^n}=\begin{cases}
0 & \text{if $2\mid n$ or $p\equiv 1\bmod{4}$,} \\
\pi/2 & \text{if $2\nmid n$ and $p\equiv 3\bmod{4}$}
\end{cases}
\end{equation*}
and $s$ is any solution of
\begin{equation*}
s^2\equiv a \bmod p^n.
\end{equation*}
\item
The bound
\begin{equation}
\left\vert\mathsf{Kl}_{p^n}(a,1)\right\vert\leq 2
\end{equation}
holds.
\item
Let $a$ be a non-zero square modulo $p$. There exists some integers $c_0^\prime, \dots c_{n-1}^\prime$ satisfying
\begin{equation}\label{eq_vp_cm'}
\forall m\in\{0,\dots,n-1\},\quad c_m'\neq 0\quad\text{ and }\quad \nu_p(c_m')=0
\end{equation}  
and some integers $k$ and $b\in\{0,\dots,(p-1)/2\}$ depending on $a$ and $p$ so that
\begin{equation}\label{eq_sr}
s_{a,p^n}=b\sum_{m=0}^{n-1}c_m^\prime\overline{b}^{2m}p^mk^m
\end{equation}
is a solution of $s^2\equiv a \bmod p^n$, where $\overline{b}$ stands for the inverse of $b$ modulo $p^n$.
\end{itemize}
\end{lemma}
\begin{proof}[\proofname{} of lemma \ref{lemma_kppm}]%
The three first items are standard. See \cite[Chapter 12 Equation (12.39)]{IwKo}. In particular, recall that $a$ is a non-zero square modulo $p$ is equivalent to saying that $a$ is a non-zero square modulo $p^n$.
\par
Let us consider the last one. The elements of $\left(\left(\Z/p\Z\right)^\times\right)^2$ are given by
\begin{equation*}
b^2,\quad 1\leq b\leq (p-1)/2.
\end{equation*}
Thus,
\begin{equation*}
a\equiv b^2\bmod{p}
\end{equation*}
for some $1\leq b=b_{a,p}\leq (p-1)/2$ so that
\begin{equation*}
a=b^2+pk
\end{equation*}
for some $k=k_{a,b,p}$ in $\Z$. The congruence to be solved becomes
\begin{equation*}
s^2\equiv a=b^2+pk\equiv b^2\left(1+\overline{b}^2pk\right)\bmod p^n
\end{equation*}
where $\overline{b}$ stands for a representative of the inverse of $b$ modulo $p^n$.  Let us define the $p$-adic integers\footnote{Recall that the prime $p$ is odd.} $c_0=1$, $c_1=1/2$ and
\begin{equation*}
\forall m\geq 2,\quad c_m\coloneqq\frac{(-1)^{m-1}(2m-3)!}{2^{2(m-1)}m!(m-2)!}=\frac{1/2(1/2-1)\dots(1/2-m+1)}{m!}.
\end{equation*}
Obviously,
\begin{equation*}
\forall m\geq 1,\quad c_m=-\frac{2m-3}{2m}c_{m-1}
\end{equation*}
so that
\begin{equation}\label{eq_vp_cm}
\forall m\in\{0,\dots,n-1\},\quad\nu_p(c_m)=0
\end{equation}
since $p\geq 2n-5$. If $x\in p\Z_p$ then, by \cite[Chapter IV.1]{MR754003}, the power series
\begin{equation*}
\sum_{m\geq 0}c_mx^m\in\Z_p[[x]] 
\end{equation*}
converges in the $p$-adic norm to a square root of $1+x$. As a consequence, one has
\begin{equation*}
s\equiv b\sum_{m=0}^{n-1}c_m^\prime\overline{b}^{2m}p^mk^m\bmod{p^n}
\end{equation*}
where the coefficients $c_m^\prime$ are some integers satisfying $c_0^\prime=1$ and
\begin{equation*}
\forall m\geq 1,\quad c_m^\prime\equiv c_m\bmod{p^n},\quad 0\leq c_m^\prime<p^n.
\end{equation*}
In particular, if $0\leq m\leq n-1$ then 
\begin{equation*}
c_m'\neq 0\quad\text{ and }\quad \nu_p(c_m')=\nu_p(c_m)=0.
\end{equation*}
by \eqref{eq_vp_cm}.
\end{proof}
The following proposition contains the upper-bound for $\mathsf{N}(\uple{\mu},\uple{\ell};w)$ defined in \eqref{eq_tricky_counting}.
\begin{proposition}[A counting argument]\label{propo_tricky_counting}
Let $\uple{\mu}=\left(\mu(\tau)\right)_{\tau\in\mathbb{Z}/p^n\mathbb{Z}}$ be a sequence of $p^n$-tuples of non-negative integers satisfying
$\abs{\mathsf{T}(\uple{\mu})}=\abs{\overline{\mathsf{T}}(\uple{\mu})}$ and $\uple{\ell}$ a $\abs{\mathsf{T}(\uple{\mu})}$-tuple of integers satisfying
\begin{equation*}
\forall\tau\in\mathsf{T}(\uple{\mu}),\quad\abs{\ell_\tau}<p
\end{equation*}
and $\uple{\ell}\neq\uple{0}$. One uniformly has
\begin{equation*}
\mathsf{N}(\uple{\mu},\uple{\ell};w)\ll_{\vert\mathsf{T}(\uple{\mu})\vert}1
\end{equation*}
for any $w\bmod{p}$ where the implied constant only depends on $\vert\mathsf{T}(\uple{\mu})\vert$.
\end{proposition}
\begin{proof}[\proofname{} of proposition \ref{propo_tricky_counting}]%
Let $k\coloneqq\vert\mathsf{T}(\uple{\mu})\vert$ for simplicity.
\par
Let us assume that $k=1$. In this case, $\mathsf{T}(\uple{\mu})=\{\tau_0\}$ and one has
\begin{equation*}
\ell_{\tau_0}\overline{b_{\tau_0}}=m_{\uple{b},\uple{\ell}}(1,1)\equiv w\bmod{p},
\end{equation*}
which fixes the value of $b_{\tau_0}$ since $\ell_{\tau_0}$ is coprime with $p$.
\par
Let us assume from now on that $k\geq 2$. One has
\begin{equation}\label{eq_c}
\mathsf{N}(\uple{\mu},\uple{\ell};w)=\sum_{\substack{c\bmod{p} \\
(p,c)=1}}\sum_{\substack{\uple{b}\in\mathsf{B}_{p^n}(\uple{\mu}) \\
m_{\uple{b}}(1,1)\equiv w\bmod{p} \\
\forall j\in\{2,\dots,n-1\},\quad m_{\uple{b}}(j,j)\equiv 0\bmod{p} \\
\forall\tau\in\mathsf{T}(\uple{\mu}),\quad b_\tau^2\equiv c+\tau\bmod{p}}}1.
\end{equation}
Note that for a fixed $c$, there is at most one tuple $\uple{b}$ since their coordinates satisfy the given quadratic equations modulo $p$. The basic idea to show that there is a bounded number of integers $c$ modulo $p$ is to find a polynomial, which vanishes on these $c$'s and whose degree only depends on $k$.
\par
Let us consider the polynomial
\begin{equation*}
Q(\uple{a};X)=\prod_{\uple{\epsilon}=\left(\epsilon_\tau\right)_{\tau\in\mathsf{T}(\uple{\mu})}\in\{\pm 1\}^k}\left(X-\sum_{\tau\in\mathsf{T}(\uple{\mu})}\epsilon_\tau a_\tau\right)\in\mathbb{F}_p[\uple{a},X]
\end{equation*}
in the variables $a_\tau, \tau\in\mathsf{T}(\uple{\mu})$, and $X$.
\par
This polynomial can be written as
\begin{equation*}
Q(\uple{a};X)=\sum_{i=0}^{2^{k-1}}Q_i\left(\uple{a}\right)X^{2i}+X^{2^k}
\end{equation*}
where $Q_i\in\mathbb{F}_p\left[\uple{a}\right]$ is a homogeneous polynomial of degree $2^k-2i$ for $0\leq i\leq 2^{k-1}$, which only involves even powers of $a_i$ ($0\leq i\leq k$). The fact that only even powers of $X$ occur easily follows from the fact that if $\uple{\epsilon}$ belongs to $\{\pm 1\}^k$ then so does $-\uple{\epsilon}$. The fact that each monomial only contains even powers of $a_i$ for $1\leq i\leq k$ is due to the obvious invariance property given by
\begin{equation*}
\forall\uple{\epsilon}\in\{\pm 1\}^k,\quad Q(\uple{\epsilon}.\uple{a};X)=Q(\uple{a};X)
\end{equation*}
where $.$ stands for the coordinates by coordinates product between tuples.
\par
The previous discussion implies that
\begin{equation*}
R_{\uple{\ell}}(\uple{Y};X)\coloneqq\left(\prod_{\tau\in\mathsf{T}(\uple{\mu})}Y_\tau\right)^{2^k}Q(\uple{\ell}.\uple{Y}^{-1};X)=\sum_{i=0}^{2^{k-1}}R_{i,\uple{\ell}}\left(\uple{Y}\right)X^{2i}+\left(\prod_{\tau\in\mathsf{T}(\uple{\mu})}Y_\tau\right)^{2^k}X^{2^k}
\end{equation*}
where $\uple{Y}=\left(Y_\tau\right)_{\tau\in\mathsf{T}(\uple{\mu})}$ and $\uple{Y}^{-1}=\left(Y_\tau^{-1}\right)_{\tau\in\mathsf{T}(\uple{\mu})}$ and for $0\leq i\leq 2^{k-1}$, $R_{i,\uple{\ell}}\in\mathbb{F}_p\left[\uple{Y}^2\right]$ is a homogeneous polynomial of degree $(k-1)2^k+2i$, which only involves even powers of $Y_\tau$ for $\tau\in\mathsf{T}(\uple{\mu})$. Here, $\uple{Y}^2=\left(Y_\tau^2\right)_{\tau\in\mathsf{T}(\uple{\mu})}$.
\par
Let us denote by $\psi$ the ring morphism from $\mathbb{F}_p\left[\uple{Y}^2\right]$ to $\mathbb{F}_p[Z]$ defined by
\begin{equation*}
\forall\tau\in\mathsf{T}(\uple{\mu}),\quad\psi\left(Y_\tau^2\right)=Z+\tau.
\end{equation*}
\par
Let us assume that $(p,w)=1$. Note that if the tuple $\uple{b}$ satisfies the constraints given in \eqref{eq_c} then $R_{\uple{\ell}}(\uple{b};w)=0$ since the contribution of $\uple{\epsilon}=(1,\dots,1)$ in $Q(\uple{\ell}.\uple{b}^{-1};w)$ is exactly $w-m_{\uple{b},\uple{\ell}}(1,1)\equiv 0\bmod{p}$. Thus, $c$ is a root of the polynomial $\psi(R_{\uple{\ell}}(\uple{Y};w))$, which is of degree $k2^{k-1}$ and leading coefficient $w^{2^k}\neq 0\bmod{p}$. As a consequence, the number of $c$'s in \eqref{eq_c} is less than $k2^{k-1}$.
\par
Let us assume that $w\equiv 0\bmod{p}$. Let $\tau_0$ in $\mathsf{T}(\uple{\mu})$ satisfying
\begin{equation*}
p\nmid\ell_{\tau_0},
\end{equation*}
which exists by the conditions of the tuple $\uple{\ell}$.
\par
Let us consider
\begin{equation*}
S_{\uple{\ell}}(\uple{Y})\coloneqq\left(\prod_{\tau\in\mathsf{T}(\uple{\mu})}Y_\tau\right)^{2^{k-1}}Q\left(\uple{\tilde{\ell}}.\uple{\tilde{Y}}^{-1};\ell_{\tau_0}Y_{\tau_0}^{-1}\right)\in\mathbb{F}_p[\uple{\tilde{Y}},Y_{\tau_0}]
\end{equation*}
where $\uple{\tilde{Y}}^{-1}=\left(Y_\tau^{-1}\right)_{\tau\in\mathsf{T}(\uple{\mu})\setminus\{\tau_0\}}$ and  $\uple{\tilde{\ell}}=\left(\ell_\tau\right)_{\tau\in\mathsf{T}(\uple{\mu})\setminus\{\tau_0\}}$ . This polynomial is homogenous of degree $(k-1)2^{k-1}$ and only involves even powers of $Y_\tau$ for $\tau\in\mathsf{T}(\uple{\mu})$. 
\par
Thus, the polynomial $U=\psi(S_{\uple{\ell}}(\uple{Y}))\in\mathbb{F}_p(Z)$ is of degree less than $(k-1)2^{k-2}$. Let us show that this polynomial is of degree at least one. If not, all the coefficients but the constant one of the polynomial $U$ vanish. If the tuple $\uple{b}$ satisfies the constraints given in \eqref{eq_c} then $S_{\uple{\ell}}(\uple{b})=0$ because of the contribution of $\uple{\epsilon}=(-1,\dots,-1)$ in $Q\left(\uple{\tilde{\ell}}.\uple{\tilde{Y}}^{-1};\ell_{\tau_0}Y_{\tau_0}^{-1}\right)$. This implies that $U(c)=0$ and that $U$ is the constant polynomial of value $0$. Choosing $Z=-\tau_{0}$ leads to
\begin{equation*}
\ell_{\tau_0}^{2^{k-1}}\prod_{\substack{\tau\in\mathsf{T}(\uple{\mu}) \\
\tau\neq\tau_0}}(\tau-\tau_{0})^{2^{k-1}}\equiv 0\bmod{p}
\end{equation*} 
so that
\begin{equation*}
\ell_{\tau_0}\equiv 0\bmod{p}
\end{equation*}
since the $\tau$'s are distinct modulo $p$. This is a contradiction.
\par
Finally, the $c$'s satisfy the polynomial equation $U(c)=0$ of degree at least $1$ and less than $(k-1)2^{k-2}$.  As a consequence, the number of $c$'s in \eqref{eq_c} is less than $(k-1)2^{k-2}$.
\end{proof}
\par
The following proposition contains the asymptotic evaluation of the cardinality of the set $\mathsf{A}_{p^n}(\uple{\mu})$ defined in \eqref{eq_apmu}.
\begin{proposition}[Applying A.~Weil's version of the Riemann hypothesis]\label{propo_cardinality}
Let $\uple{\mu}=\left(\mu(\tau)\right)_{\tau\in\mathbb{Z}/p^n\mathbb{Z}}$ be a sequence of $p^n$-tuples of non-negative integers. If $p$ is odd then
\begin{equation}\label{eq_cardinality}
\left\vert\mathsf{A}_{p^n}(\uple{\mu})\right\vert=\frac{\varphi(p^n)}{2^{\abs{\overline{\mathsf{T}}(\uple{\mu})}}}\left(1+O\left(\frac{2^{\abs{\overline{\mathsf{T}}(\uple{\mu})}}\abs{\overline{\mathsf{T}}(\uple{\mu})}}{p^{1/2}}\right)\right).
\end{equation}
\end{proposition}
\begin{remark}
The equation \eqref{eq_cardinality} is an asymptotic expansion if and only if
\begin{equation}\label{eq_cond_asymp}
\frac{2^{\abs{\overline{\mathsf{T}}(\uple{\mu})}}\abs{\overline{\mathsf{T}}(\uple{\mu})}}{p^{1/2}}\to 0
\end{equation}
as $p$ tends to infinity among the prime numbers.
\end{remark}
\begin{proof}[\proofname{} of proposition \ref{propo_cardinality}]%
Obviously,
\begin{align*}
\left\vert\mathsf{A}_{p^n}(\uple{\mu})\right\vert & =p^{n-1}\sum_{\substack{a\in\left(\Z/p\Z\right)^\times \\
\forall\tau\in\mathsf{T}(\uple{\mu}), a+\tau\in\left(\left(\Z/p\Z\right)^\times\right)^2}}1 \\
& =p^{n-1}\sum_{\substack{a\in\left(\Z/p\Z\right)^\times \\
\forall t\in\overline{\mathsf{T}}(\uple{\mu}), a+t\in\left(\left(\Z/p\Z\right)^\times\right)^2}}1 \\
& =p^{n-1}\sum_{a\in\left(\Z/p\Z\right)^\times}\prod_{\substack{t\in\overline{\mathsf{T}}(\uple{\mu}) \\
(a+t,p)=1}}\frac{1}{2}\left(\chi_2(a+t)+1\right) \\
\end{align*}
where $\chi_2$ is the quadratic character modulo the odd prime number $p$. 
\par
At this point, the problem becomes a variant of the question considered by H.~Davenport in 1931 of counting elements $x$ modulo $p$ such that both $x, x+1,\dots,x+k$ are quadratic residues modulo $p$ uniformly with respect to the integer $k\geq 1$. See for instance \cite[Section 1.4.2]{MR617009}. Thus, the end of the proof is omitted. 
\end{proof}
The core of the proof of proposition \ref{propo_moment_tilde} is the following result.
\begin{proposition}[Moments of shifted Kloosterman sums]\label{propo_shifted_Kl}
Let $\uple{\mu}=\left(\mu(\tau)\right)_{\tau\in\mathbb{Z}/p^n\mathbb{Z}}$ be a sequence of $p^n$-tuples of non-negative integers satisfying
\begin{equation}\label{eq_assump_mu}
\sum_{\tau\in\Z/p^n\Z}\mu(\tau)\leq M
\end{equation}
for some absolute positive constant $M$ and $\abs{\mathsf{T}(\uple{\mu})}=\abs{\overline{\mathsf{T}}(\uple{\mu})}$. If 
\begin{equation}\label{eq_assump_p}
p>\max{\left(M,2n-5\right)}
\end{equation}
then
\begin{equation}\label{eq_asymp_S}
\mathsf{S}_{p^n}(\uple{\mu};b_0)=\left[\prod_{\tau\in \mathbb{Z}/p^n\mathbb{Z}}\delta_{2\mid\mu(\tau)}\binom{\mu(\tau)}{\mu(\tau)/2}\right]\frac{\left\vert\mathsf{A}_{p^n}(\uple{\mu})\right\vert}{\varphi(p^n)}+O_{M,\epsilon}\left(p^{-\frac{4(n-1)}{2^n}+\epsilon}\right)
\end{equation}
for any $\epsilon>0$ and where the implied constant only depends on $M$ and $\epsilon$.
\end{proposition}
\begin{remark}
In particular, for any non-negative integer $m$,
\begin{equation}\label{eq_usual_moment}
\frac{1}{\varphi(p^n)}\sum_{a\in\left(\mathbb{Z}/p^n\mathbb{Z}\right)^\times}\mathsf{Kl}_{p^n}(a,b_0)^m=\delta_{2\mid m}\frac{1}{2}\binom{m}{m/2}+O_{m,\epsilon}\left(p^{-\frac{4(n-1)}{2^n}+\epsilon}\right)
\end{equation}
for any $\epsilon>0$ under \eqref{eq_assump_p}. In other words, under the same assumption,
\begin{equation*}
\frac{1}{\varphi(p^n)}\sum_{a\in\left(\mathbb{Z}/p^n\mathbb{Z}\right)^\times}\mathsf{Kl}_{p^n}(a,b_0)^m=\mathbb{E}(U^m)+O_{m,\epsilon}\left(p^{-\frac{4(n-1)}{2^n}+\epsilon}\right)
\end{equation*}
where $U$ is any real-valued random variable of law the probability measure $\mu$ defined in \eqref{eq_def_mu_1}. Hence, by \eqref{eq_mu_1}, the normalized Kloosterman sums $\mathsf{Kl}_{p^n}(a,b_0)$ become equidistributed in $[-2,2]$ with respect to the measure $\mu$ as briefly indicated in Remark \ref{remark_111}. Such equidistribution result was stated without proof in \cite[Remark 1.1]{MR2646758}. This measure has already occured in \cite{MR2646758}, where the author proves that the twisted normalized Kloosterman sums $\mathsf{Kl}_{p^n}(a,\chi)$ for a fixed $a$ in $\Z/p^n\Z$ and $\chi$ ranging over the Dirichlet characters of modulus $p^n$ get equidistributed with respect to $\mu$ as $p$ tends to infinity.
\end{remark}
\begin{remark}
It follows from the results proved in \cite{GoVeWa} that if $1\leq m\leq p^{n-1}$ then
\begin{equation*}
\frac{1}{\varphi(p^n)}\sum_{a\in\left(\mathbb{Z}/p^n\mathbb{Z}\right)^\times}\mathsf{Kl}_{p^n}(a,1)^m=\mathbb{E}(U^m)
\end{equation*}
where $U$ is any real-valued random variable of law the probability measure $\mu$, which agrees with \eqref{eq_usual_moment}.
\end{remark}
\begin{remark}
For any integer $r\geq 1$, any non-negative integers $m_1,\dots,m_r$ and any distinct integers $\tau_1,\dots,\tau_r$, the previous proposition implies that
\begin{equation*}
\frac{1}{\varphi(p^n)}\sum_{a\in\left(\mathbb{Z}/p^n\mathbb{Z}\right)^\times}\prod_{i=1}^r\mathsf{Kl}_{p^n}(a+\tau_i,b_0)^{m_i}=\mathbb{E}\left(\prod_{i=1}^rU_i^{m_i}\right)+O_{m_1+\dots+m_r,\epsilon}\left(p^{-\frac{4(n-1)}{2^n}+\epsilon}\right)
\end{equation*}
for any $\epsilon>0$ and for any sequence of real-valued independent random variables $\left(U_i\right)_{1\leq i\leq r}$ of law the probability measure $\mu$ under \eqref{eq_assump_p} provided that
\begin{equation*}
p>\max_{1\leq i,j\leq r}{\abs{\tau_i-\tau_j}}.
\end{equation*}
In other words, the $r$-tuple $\left(\mathsf{Kl}_{p^n}(a+\tau_i,b_0)\right)_{1\leq i\leq r}$ gets equidistributed in $[-2,2]^r$ with respect to the measure $\otimes^r\mu$.
\end{remark}
\begin{proof}[\proofname{} of proposition \ref{propo_shifted_Kl}]%
Firstly,
\begin{equation*}
\mathsf{Kl}_{p^n}(a+\tau,b_0)=\mathsf{Kl}_{p^n}(b_0a+b_0\tau,1)
\end{equation*}
since $b_0$ is coprime with $p$. The change of variable $a'=b_0a$ in $\mathsf{S}_{p^n}(\uple{\mu};b_0)$ combined with the change of multiplicities
\begin{equation*}
\mu(\tau)=\begin{cases}
\mu(\tau') & \text{if $\tau=b_0\tau'$,} \\
0 & \text{otherwise}
\end{cases}
\end{equation*}
for $\tau\in\Z/p^n\Z$ implies that one has to prove this proposition only for $b_0=1$. Thus, $b_0=1$ up to the completion of the proof.
\par
Let us come back to the moment $\mathsf{S}_{p^n}(\uple{\mu})$.
%
By Lemma \ref{lemma_kppm},
\begin{multline*}
\mathsf{S}_{p^n}(\uple{\mu})=\frac{1}{\varphi(p^n)}\sum_{\uple{b}\in\mathsf{B}_{p^n}(\uple{\mu})}\prod_{\tau\in\mathsf{T}(\uple{\mu})}\left(\frac{b_\tau}{p^n}\right)^{\mu(\tau)} \\
\sum_{\substack{a\in\Z/p^n\Z \\
\forall\tau\in\mathsf{T}(\uple{\mu}), a\equiv b_\tau^2-\tau\bmod{p}}}\prod_{\tau\in\mathsf{T}(\uple{\mu})}\left(2\cos{\left(\frac{4\pi s_{a+\tau,p^n}}{p^n}+\theta_{p^n}\right)}\right)^{\mu(\tau)}.
\end{multline*}
Recall that $s_{a+\tau,p^n}^2\equiv a+\tau\bmod{p^n}$. In addition, the second condition in \eqref{eq_deuxieme} is satisfied since $a$ has to be coprime with $p$. 
\par
Now, recall that\footnote{The referee kindly informed us that this expansion can be interpreted as an expansion in terms of Chebychev polynomials of the first kind, which are orthogonal polynomials for the measure $\mu$.}
\begin{equation*}
\left(2\cos{(x)}\right)^M=\sum_{m=0}^M\binom{M}{m}\cos{\left((M-2m)x\right)}
\end{equation*}
for any real number $x$ and any non-negative integer $M$. Thus,
\begin{multline}\label{eq_propre}
\mathsf{S}_{p^n}(\uple{\mu})=\frac{1}{\varphi(p^n)}\sum_{\uple{b}\in\mathsf{B}_{p^n}(\uple{\mu})}\prod_{\tau\in\mathsf{T}(\uple{\mu})}\left(\frac{b_\tau}{p^n}\right)^{\mu(\tau)} \\
\sum_{\substack{a\in\Z/p^n\Z \\
\forall\tau\in\mathsf{T}(\uple{\mu}), a\equiv b_\tau^2-\tau\bmod{p}}}\prod_{\tau\in\mathsf{T}(\uple{\mu})}\sum_{u_\tau=0}^{\mu(\tau)}\binom{\mu(\tau)}{u_\tau}\cos{\left[\left(\mu(\tau)-2u_\tau\right)\left(\frac{4\pi s_{a+\tau,p^n}}{p^n}+\theta_{p^n}\right)\right]}.
\end{multline}
\par
One can split $\mathsf{S}_{p^n}(\uple{\mu})$ into $\mathsf{S}_{p^n}(\uple{\mu})=\mathsf{MT}_{p^n}(\uple{\mu})+\mathsf{Err}_{p^n}(\uple{\mu})$ where
\begin{multline*}
\mathsf{MT}_{p^n}(\uple{\mu})\coloneqq\frac{1}{\varphi(p^n)}\sum_{\uple{b}\in\mathsf{B}_{p^n}(\uple{\mu})}\prod_{\tau\in\mathsf{T}(\uple{\mu})}\left(\frac{b_\tau}{p^n}\right)^{\mu(\tau)} \\
\sum_{\substack{a\in\Z/p^n\Z \\
\forall\tau\in\mathsf{T}(\uple{\mu}), a\equiv b_\tau^2-\tau\bmod{p}}}\prod_{\tau\in\mathsf{T}(\uple{\mu})}\sum_{2u_\tau=\mu(\tau)}\binom{\mu(\tau)}{u_\tau}\cos{\left[\left(\mu(\tau)-2u_\tau\right)\left(\frac{4\pi s_{a+\tau,p^n}}{p^n}+\theta_{p^n}\right)\right]}
\end{multline*}
and $\mathsf{Err}_{p^n}(\uple{\mu})$ is the remaining term. Note that $\mathsf{MT}_{p^n}(\uple{\mu})$ is nothing else than the term obtained when the multiplicities $\mu(\tau)$ are even and $u_\tau=\mu(\tau)/2$.
\par
Let us start with $\mathsf{MT}_{p^n}(\uple{\mu})$. Obviously, $\mathsf{MT}_{p^n}(\uple{\mu})=0$ unless
\begin{equation*}
\forall\tau\in\mathsf{T}(\uple{\mu}),\quad 2\mid\mu(\tau).
\end{equation*}
Hence
\begin{equation*}
\mathsf{MT}_{p^n}(\uple{\mu})=\left[\prod_{\tau\in\mathsf{T}(\uple{\mu})}\delta_{2\mid\mu(\tau)}\binom{\mu(\tau)}{\mu(\tau)/2}\right]\frac{\left\vert\mathsf{A}_{p^n}(\uple{\mu})\right\vert}{\varphi(p^n)}.
\end{equation*}
\par
Let us bound $\mathsf{Err}_{p^n}(\uple{\mu})$. Trivially,
\begin{equation*}
\mathsf{Err}_{p^n}(\uple{\mu})\ll_M\sup_{\substack{\uple{\ell}\in\prod_{\tau\in\mathsf{T}(\uple{\mu})}[-\mu(\tau),\mu(\tau)] \\
\uple{\ell}\neq\uple{0}}}\mathsf{Err}_{p^n}(\uple{\mu},\uple{\ell})
\end{equation*}
where
\begin{align*}
\mathsf{Err}_{p^n}(\uple{\mu},\uple{\ell}) & =\frac{1}{\varphi(p^n)}\sum_{\uple{b}\in\mathsf{B}_{p^n}(\uple{\mu})}\left\vert\sum_{\substack{a\in\Z/p^n\Z \\
\forall\tau\in\mathsf{T}(\uple{\mu}), a\equiv b_\tau^2-\tau\bmod{p}}}e_{p^n}\left(\sum_{\tau\in\mathsf{T}(\uple{\mu})}\ell_\tau s_{a+\tau,p^n}\right)\right\vert \\
& \coloneqq\frac{1}{\varphi(p^n)}\sum_{\uple{b}\in\mathsf{B}_{p^n}(\uple{\mu})}\left\vert\mathsf{Err}_{p^n}(\uple{\mu},\uple{\ell},\uple{b})\right\vert.
\end{align*}
for any $\abs{\mathsf{T}(\uple{\mu})}$-tuple $\uple{\ell}$ of integers with the properties written above.
\par
Let us fix from now on a $\abs{\mathsf{T}(\uple{\mu})}$-tuple $\uple{\ell}=\left(\ell_\tau\right)_{\tau\in\mathsf{T}(\uple{\mu})}$ of integers different from the tuple $\uple{0}$ and satisfying
\begin{equation}\label{eq_assump_l}
\forall\tau\in\mathsf{T}(\uple{\mu}),\quad\abs{\ell_\tau}\leq\mu(\tau)\leq M<p
\end{equation}
by \eqref{eq_assump_mu}. Let $\tau_0$ be a fixed element of $\mathsf{T}(\uple{\mu})$. For $\uple{b}$ in $\mathsf{B}_{p^n}(\uple{\mu})$, \eqref{eq_asump_1} implies that 
\begin{equation*}
\forall\tau\in\mathsf{T}(\uple{\mu}),\exists d_\tau\in\Z,\quad b_\tau^2-\tau=b_{\tau_0}^2-\tau_0-d_\tau p.
\end{equation*}
One gets a $\vert\mathsf{T}(\uple{\mu})\vert$-tuple $\uple{d}=\left(d_\tau\right)_{\tau\in\mathsf{T}(\uple{\mu})}$ of integers. The change of variables 
\begin{equation*}
a=b_{\tau_0^2}-\tau_0+up
\end{equation*}
with $u\bmod{p^{n-1}}$ so that
\begin{equation*}
a+\tau=b_{\tau_0}^2-\tau_0+up+\tau=b_\tau^2+(d_\tau+u)p
\end{equation*}
and \eqref{eq_sr} entail that
\begin{align*}
\mathsf{Err}_{p^n}(\uple{\mu},\uple{\ell},\uple{b}) & =\sum_{u\bmod{p^{n-1}}}e_{p^n}\left(\sum_{\tau\in\mathsf{T}(\uple{\mu})}\ell_{\tau}s_{b_\tau^2+(d_\tau+u)p,p^n}\right) \\
& =\sum_{u\bmod{p^{n-1}}}e_{p^n}\left(\sum_{\tau\in\mathsf{T}(\uple{\mu})}\ell_{\tau}b_\tau\sum_{m=0}^{n-1}c_m^\prime\overline{b_\tau}^{2m}p^m(d_\tau+u)^m\right) \\
& =\sum_{u\bmod{p^{n-1}}}e_{p^n}\left(P_{\uple{b}}(u)\right) \\
\end{align*}
where
\begin{equation}\label{eq_pol}
P_{\uple{b}}(u)\coloneqq\sum_{\tau\in\mathsf{T}(\uple{\mu})}\ell_{\tau}b_\tau\sum_{m=0}^{n-1}c_m^\prime\overline{b_\tau}^{2m}p^m(d_\tau+u)^m
\end{equation}
is a polynomial in the variable $u$ of degree less than $n-1$ with integer coefficients. Note that all the quantities defined here and below depend on the tuples $\uple{\mu}$, $\uple{\ell}$, $\uple{b}$ and $\uple{d}$ but we only state the dependence on $\uple{b}$ for simplicity. One can check that
\begin{equation*}
P_{\uple{b}}(u)=\sum_{j=0}^{n-1}a_j(\uple{b})u^j
\end{equation*}
where
\begin{equation*}
\forall j\in\{0,\dots,n-1\},\quad a_j(\uple{b})=\sum_{r=j}^{n-1}\binom{r}{j}c_r'm_{\uple{b}}(r,j)p^r
\end{equation*}
and
\begin{equation*}
\forall j\in\{0,\dots,n-1\}, \forall r\in\{j,\dots,n-1\},\quad m_{\uple{b}}(r,j)=\sum_{\tau\in\mathsf{T}(\uple{\mu})}\ell_\tau\overline{b_\tau}^{2r-1}d_\tau^{r-j}.
\end{equation*}
An important fact is that $p^j$ divides $a_j(\uple{b})$ for any $j\in\{0,\dots,n-1\}$. Note also that when $r=j$, the quantity $m_{\uple{b}}(r,j)$ gets simpler and does not depend on the tuple $\uple{d}$ since $m_{\uple{b}}(j,j)=m_{\uple{b},\uple{\ell}}(j,j)$ previously defined in \eqref{eq_mmm} for $1\leq j\leq n-1$. In particular,
\begin{equation*}
m_{\uple{b}}(1,1)=\sum_{\tau\in\mathsf{T}(\uple{\mu})}\ell_\tau\overline{b_\tau}.
\end{equation*} 
Let us define
\begin{equation*}
j(\uple{b})\coloneqq\sup{\left(\left\{j\in\{1,\dots,n-1\}, p^n\nmid a_j(\uple{b})\right\}\right)}\in\{1,\dots,n-1\}\cup\{-\infty\}.
\end{equation*}
Having this notation in mind,
\begin{equation*}
\mathsf{Err}_{p^n}(\uple{\mu},\uple{\ell},\uple{b})=\sum_{u\bmod{p^{n-1}}}e_{p^n}\left(\sum_{j=0}^{n-1}a_j(\uple{b})u^j\right).
\end{equation*}
This new polynomial in the exponential sum is still denoted by $P_{\uple{b}}(u)$ for simplicity, even though some terms are missing.
\par
The strategy to find an upper-bound for $\mathsf{Err}_{p^n}(\uple{\mu},\uple{\ell})$ is to decompose it into
\begin{multline}\label{eq_strategy}
\mathsf{Err}_{p^n}(\uple{\mu},\uple{\ell})=\frac{1}{\varphi(p^n)}\sum_{\substack{\uple{b}\in\mathsf{B}_{p^n}(\uple{\mu}) \\
j(\uple{b})=-\infty}}\left\vert\mathsf{Err}_{p^n}(\uple{\mu},\uple{\ell},\uple{b})\right\vert+\frac{1}{\varphi(p^n)}\sum_{\substack{\uple{b}\in\mathsf{B}_{p^n}(\uple{\mu}) \\
j(\uple{b})=1}}\left\vert\mathsf{Err}_{p^n}(\uple{\mu},\uple{\ell},\uple{b})\right\vert \\
+\frac{1}{\varphi(p^n)}\sum_{\substack{\uple{b}\in\mathsf{B}_{p^n}(\uple{\mu}) \\
j(\uple{b})\in\{2,\dots,n-1\}}}\left\vert\mathsf{Err}_{p^n}(\uple{\mu},\uple{\ell},\uple{b})\right\vert
\end{multline}
and to proceed as follows.
\begin{itemize}
\item
In the first term of \eqref{eq_strategy}, the exponential sum $\mathsf{Err}_{p^n}(\uple{\mu},\uple{\ell},\uple{b})$ is bounded trivially by $p^{n-1}$ but the counting of the tuples $\uple{b}$ is done carefully;
\item
In the third term of \eqref{eq_strategy}, Weyl's differencing process enables us to find an upper-bound for the exponential sum $\mathsf{Err}_{p^n}(\uple{\mu},\uple{\ell},\uple{b})$ and the counting of the tuples $\uple{b}$ is done trivially by $\ll p$. Note that this term only occurs if $n\geq 3$.
\item
In the second term of \eqref{eq_strategy}, both the exponential sum $\mathsf{Err}_{p^n}(\uple{\mu},\uple{\ell},\uple{b})$ and the counting of the tuples $\uple{b}$ are handled carefully.
\end{itemize}
Let us define $N=p^{n-1}$ for simplicity.
\par
Let us begin with the third term of \eqref{eq_strategy}. The purpose is to show that if $\uple{b}\in\mathsf{B}_{p^n}(\uple{\mu})$ with $j(\uple{b})\in\{2,\dots,n-1\}$ then
\begin{equation}\label{eq_third}
\mathsf{Err}_{p^n}(\uple{\mu},\uple{\ell},\uple{b})\ll_\epsilon p^{n-1-\frac{4(n-1)}{2^n}+\epsilon}
\end{equation}
for any $\epsilon>0$ and where the implied constant only depends on $\epsilon$. For these tuples $\uple{b}$, $P_{\uple{b}}(u)$ is a polynomial of degree $j(\uple{b})$ and leading coefficient divisible by $p^{j(\uple{b})}$. Let us define  $a_{j(\uple{b})}(\uple{b})=p^{k}\alpha_{j(\uple{b})}$ where $j(\uple{b})\leq k\leq n-1$ and $p\nmid \alpha_{j(\uple{b})}$. We are tempted to apply Weyl's differencing process (see \cite{Wey}). By \cite[Proposition 8.2]{IwKo}), one gets
\begin{equation}\label{eq_Weyl_n}
\left\vert\mathsf{Err}_{p^n}(\uple{\mu},\uple{\ell},\uple{b})\right\vert\leq 2N\times E^{2^{1-j(\uple{b})}}
\end{equation}
where
\begin{equation*}
E\coloneqq\frac{1}{N^{j(\uple{b})}}\sum_{-N<\ell_1,\dots,\ell_{j(\uple{b})-1}<N}\min{\left(N,\left\vert\left\vert\frac{\alpha_{j(\uple{b})}j(\uple{b})!\ell_1\dots\ell_{j(\uple{b})-1}}{p^{n-j(\uple{b})}}\right\vert\right\vert^{-1}\right)}.
\end{equation*}
As usual, $\vert\vert\ast\vert\vert$ stands for the distance to the nearest integer. The contribution to $\Sigma_{j(\uple{b})}$ of the integers satisfying $\ell_1\dots\ell_{j(\uple{b})-1}=0$ is trivially bounded by $1/N$. Up to this error term,
\begin{align*}
E & = \frac{1}{N^{j(\uple{b})}}\sum_{0\neq\abs{\ell}<N^{j(\uple{b})-1}}d_{j(\uple{b})-1}(\ell)\min{\left(N,\left\vert\left\vert\frac{\alpha_{j(\uple{b})}j(\uple{b})!\ell}{p^{n-k}}\right\vert\right\vert^{-1}\right)} \\
& =\frac{1}{N^{j(\uple{b})}}\sum_{i=0}^{(n-1)(j(\uple{b})-1)-1}\sum_{\substack{0\neq\abs{\ell}<N^{j(\uple{b})-1} \\
p^i\mid\mid\ell}}d_{j(\uple{b})-1}(\ell)\min{\left(N,\left\vert\left\vert\frac{\alpha_{j(\uple{b})}j(\uple{b})!\ell}{p^{n-k}}\right\vert\right\vert^{-1}\right)} \\
& =\frac{1}{N^{j(\uple{b})}}\sum_{i=0}^{(n-1)(j(\uple{b})-1)-1}\sum_{\substack{0\neq\abs{\ell}<\frac{N^{j(\uple{b})-1}}{p^i} \\
(p,\ell)=1}}d_{j(\uple{b})-1}(p^i\ell)\min{\left(N,\left\vert\left\vert\frac{\alpha_{j(\uple{b})}j(\uple{b})!\ell}{p^{n-k-i}}\right\vert\right\vert^{-1}\right)}.
\end{align*}
The contribution to $E$ of the non-negative integers $i$ less than $n-k-1$ can be written as
\begin{multline*}
\frac{1}{N^{j(\uple{b})}}\sum_{\substack{0\leq i\leq (n-1)(j(\uple{b})-1)-1 \\
i\leq n-k-1}}\sum_{\substack{\abs{v}\leq\left(p^{n-k-i}-1\right)/2 \\
(p,v)=1}} \\
\sum_{\substack{0\neq\abs{\ell}<\frac{N^{j(\uple{b})-1}}{p^i} \\
\ell\equiv\overline{\alpha_{j(\uple{b})}j(\uple{b})!}v\bmod{p^{n-k-i}}}}d_{j(\uple{b})-1}(p^i\ell)\min{\left(N,\left\vert\left\vert\frac{v}{p^{n-k-i}}\right\vert\right\vert^{-1}\right)}
\end{multline*}
and is bounded by $(pN)^\epsilon /N$. The contribution of the remaining integers $i$ is trivially bounded by $(pN)^\epsilon/p^{n-k}$, which is less than $(pN)^\epsilon /N$. As a consequence,
\begin{equation*}
\left\vert\mathsf{Err}_{p^n}(\uple{\mu},\uple{\ell},\uple{b})\right\vert\ll_\epsilon (pN)^\epsilon N^{1-2^{1-j(\uple{b})}}\leq  (pN)^\epsilon N^{1-2^{2-n}},
\end{equation*}
which implies \eqref{eq_third}.
\par
About the first term of of \eqref{eq_strategy}, let us show that
\begin{equation}\label{eq_first}
\frac{1}{\varphi(p^n)}\sum_{\substack{\uple{b}\in\mathsf{B}_{p^n}(\uple{\mu}) \\
j(\uple{b})=-\infty}}\left\vert\mathsf{Err}_{p^n}(\uple{\mu},\uple{\ell},\uple{b})\right\vert\ll\frac{p^{n-1}}{\varphi(p^n)}\mathsf{N}(\uple{\mu},\uple{\ell};0)
\end{equation}
where $\mathsf{N}(\uple{\mu},\uple{\ell};w)$ is defined in \eqref{eq_tricky_counting} for any $w$ modulo $p$. The exponential sum in \eqref{eq_strategy} is trivially bounded by $p^{n-1}$. Now, let $\uple{b}$ in $\mathsf{B}_{p^n}(\uple{\mu})$ with $j(\uple{b})=-\infty$. If $1\leq j\leq n-1$ then $p^n$ divides $a_j(\uple{b})$, which implies that
\begin{equation*}
c_j^\prime m_{\uple{b}}(j,j)\equiv 0\bmod{p}
\end{equation*}
and $m_{\uple{b}}(j,j)\equiv 0\bmod{p}$ since $c_j^\prime$ is coprime with $p$ for $p\geq 2n-5$ by \eqref{eq_vp_cm'}. This implies \eqref{eq_first}.
\par
Finally, let us prove that
\begin{equation}\label{eq_second}
\frac{1}{\varphi(p^n)}\sum_{\substack{\uple{b}\in\mathsf{B}_{p^n}(\uple{\mu}) \\
j(\uple{b})=1}}\left\vert\mathsf{Err}_{p^n}(\uple{\mu},\uple{\ell},\uple{b})\right\vert\ll\frac{1}{\varphi(p^n)}\sum_{k=1}^{n-1}p^{n-k}\sum_{\substack{v\bmod{p^{n-k}} \\
(p,v)=1}}\frac{1}{\abs{v}}\mathsf{N}\left(\uple{\mu},\uple{\ell};\overline{c_1'}vp^{k-1}\right).
\end{equation}
For these tuples $\uple{b}$, $P_{\uple{b}}(u)$ is a polynomial of degree $1$ and leading coefficient divisible by $p$. By \cite[Equation (8.6)]{IwKo},
\begin{equation}\label{eq_Weyl_1}
\left\vert\mathsf{Err}_{p^n}(\uple{\mu},\uple{\ell},\uple{b})\right\vert\leq\frac{1}{2}\min{\left(2N,\left\vert\left\vert\frac{a_{1}(\uple{b})}{p^{n}}\right\vert\right\vert^{-1}\right)}
\end{equation}
so that
\begin{align*}
\sum_{\substack{\uple{b}\in\mathsf{B}_{p^n}(\uple{\mu}) \\
j(\uple{b})=1}}\left\vert\mathsf{Err}_{p^n}(\uple{\mu},\uple{\ell},\uple{b})\right\vert & \leq\sum_{k=1}^{n-1}\sum_{\substack{\uple{b}\in\mathsf{B}_{p^n}(\uple{\mu}) \\
p^k\mid\mid a_1(\uple{b})}}\left\vert\left\vert\frac{a_{1}(\uple{b})/p^{k}}{p^{n-k}}\right\vert\right\vert^{-1} \\
& =\sum_{k=1}^{n-1}\sum_{\substack{v\bmod{p^{n-k}} \\
(p,v)=1}}\sum_{\substack{\uple{b}\in\mathsf{B}_{p^n}(\uple{\mu}) \\
a_1(\uple{b})/p^k\equiv v\bmod{p^{n-k}}}}\left\vert\left\vert\frac{v}{p^{n-k}}\right\vert\right\vert^{-1} \\
& =\sum_{k=1}^{n-1}p^{n-k}\sum_{\substack{v\bmod{p^{n-k}} \\
(p,v)=1}}\frac{1}{\abs{v}}\sum_{\substack{\uple{b}\in\mathsf{B}_{p^n}(\uple{\mu}) \\
a_1(\uple{b})/p^k\equiv v\bmod{p^{n-k}}}}1.
\end{align*}
Now, if $\uple{b}$ in $\mathsf{B}_{p^n}(\uple{\mu})$ satisfies $a_1(\uple{b})/p^k\equiv v\bmod{p^{n-k}}$ then this implies $c_1^\prime m_{\uple{b}}(1,1)\equiv a_1(\uple{b})/p\equiv vp^{k-1}\bmod{p}$ with $c_1^\prime$ coprime with $p$ by \eqref{eq_vp_cm'}. This is exactly \eqref{eq_second}.
\par
By \eqref{eq_strategy}, \eqref{eq_third}, \eqref{eq_third} and \eqref{eq_first}, one gets
\begin{multline}\label{eq_strategy_2}
\mathsf{Err}_{p^n}(\uple{\mu},\uple{\ell})\ll_{\epsilon}p^{-\frac{4(n-1)}{2^n}+\epsilon}+\frac{\mathsf{N}(\uple{\mu},\uple{\ell};0)}{p}+\sum_{k=1}^{n-1}\frac{1}{p^k}\sum_{\substack{v\bmod{p^{n-k}} \\
(p,v)=1}}\frac{1}{\abs{v}}\mathsf{N}\left(\uple{\mu},\uple{\ell};\overline{c_1'}vp^{k-1}\right)
\end{multline}
for any $\epsilon>0$. Everything boils down to bounding $\mathsf{N}(\uple{\mu},\uple{\ell};w)$ uniformly with respect to $w\bmod{p}$. Proposition \ref{propo_tricky_counting} implies that
\begin{equation*}
\mathsf{Err}_{p^n}(\uple{\mu},\uple{\ell})\ll_{\epsilon}p^{-\frac{4(n-1)}{2^n}+\epsilon}
\end{equation*}
for any $\epsilon>0$.
\end{proof}
The following lemma will be used in the proof of Proposition \ref{propo_moment_tilde}.
\begin{lemma}\label{lemma_noncoprim}
Let $M\geq 2$ be an integer. If $a_h$ is a sequence of real numbers indexed by non-negative integers satisfying
\begin{equation*}
\forall h\in\N,\quad 0\leq a_h\leq\begin{cases}
1 & \text{if $h=0$,} \\
\frac{1}{h} & \text{otherwise}
\end{cases}
\end{equation*}
then
\begin{equation*}
\Sigma_M=\sum_{\substack{0\leq h_1,\dots, h_M\leq p^n \\
\exists i\neq j, h_{i}\equiv h_j\bmod{p} \\
h_i\neq h_j}}\prod_{i=1}^{M}a_{h_i}\ll\frac{\log^M{\left(p^n\right)}}{p}.
\end{equation*}
\end{lemma}
\begin{proof}[\proofname{} of lemma \ref{lemma_noncoprim}]%
Let us proceed by induction on $M$. If $M=2$ then
\begin{equation*}
\Sigma_2=2a_0\sum_{\substack{1\leq h\leq p^n \\
p\mid h}}a_h+\sum_{\substack{1\leq h_1, h_2\leq p^n \\
h_1\equiv h_2\bmod{p}}}a_{h_1}a_{h_2}\ll\frac{\log{\left(p^{n-1}\right)}}{p}+\frac{\log{\left(p^{n}\right)}\log{\left(p^{n-1}\right)}}{p}.
\end{equation*}
\par
Let us assume that $M\geq 3$. We use the combinatorial identity given in \cite[Lemma 7.1]{MR3248485}, which entails that
\begin{equation*}
\Sigma_M=\sum_{s=1}^M\sum_{\sigma\in P(M,s)}\sum_{\substack{0\leq h_1,\dots,h_s\leq p^n \\
\exists i\neq j, h_{i}\equiv h_j\bmod{p} \\
h_i\neq h_j}}\prod_{u=1}^{s}a_{h_u}^{\sigma_u}
\end{equation*}
where
\begin{equation*}
\forall u\in\{1,\dots,s\},\quad \sigma_u\coloneqq\left\vert\sigma^{-1}\left(\{u\}\right)\right\vert
\end{equation*}
and for $1\leq s\leq M$, $P(M,s)$ stands for the set of surjective functions
\begin{equation*}
\sigma:\left\{1,\dots,M\right\}\to\left\{1,\dots,s\right\}
\end{equation*}
satisfying
\begin{equation*}
\forall j\in\left\{1,\dots,M\right\},\quad\sigma(j)=1\quad\text{ or }\quad\exists k<j, \sigma(j)=\sigma(k)+1.
\end{equation*}
\par
The sum over $h_1,\dots,h_s$ can be decomposed into
\begin{equation*}
\sum_{\substack{0\leq h_1,\dots,h_s\leq p^n \\
h_1,\dots,h_s \text{ distinct} \\
\exists i_0, h_{i_0}=0 \\
\exists j\neq i_0, h_{j}\equiv 0\bmod{p}}}\prod_{u=1}^{s}a_{h_u}^{\sigma_u}+\sum_{\substack{0\leq h_1,\dots,h_s\leq p^n \\
h_1,\dots,h_s \text{ distinct} \\
\exists i_0, h_{i_0}=0 \\
\forall i\neq i_0, p\nmid h_i \\
\exists i\neq j\neq i_0, h_{i}\equiv h_j\bmod{p}}}\prod_{u=1}^{s}a_{h_u}^{\sigma_u}+\sum_{\substack{1\leq h_1,\dots,h_s\leq p^n \\
h_1,\dots,h_s \text{ distinct} \\
\exists i\neq j, h_{i}\equiv h_j\bmod{p}}}\prod_{u=1}^{s}a_{h_u}^{\sigma_u}.
\end{equation*}
The first sum is trivially bounded whereas the second and third sums are bounded by induction. This gives
\begin{multline*}
\Sigma_M\ll\sum_{s=1}^M\sum_{\sigma\in P(M,s)}\left(\frac{\log^{s-2}{\left(p^n\right)}\log{\left(p^{n-1}\right)}}{p}+\frac{\log^{s-2}{\left(p^n\right)}\log{\left(p^{n-1}\right)}}{p}\right. \\
\left.+\frac{\log^{s-1}{\left(p^n\right)}\log{\left(p^{n-1}\right)}}{p}\right),
\end{multline*}
which ensures the result.
\end{proof}
Let us give now the proof of Proposition \ref{propo_moment_tilde}
\begin{proof}[\proofname{} of proposition \ref{propo_moment_tilde}]%
By Lemma \ref{lemma_step_1}, it is enough to consider $\widetilde{\mathsf{M}_{p^n}}(\uple{t};\uple{m},\uple{n},b_0)$.
\par
Recall that $H_{p^n}=\left\{(1-p^n)/2,\dots,(p^n-1)/2\right\}$.
\par
By \eqref{eq_KL_tilde},
\begin{multline*}
\widetilde{\mathsf{M}_{p^n}}(\uple{t};\uple{m},\uple{n};b_0)=\frac{1}{p^{n\ell(\uple{m}+\uple{n})/2}\varphi(p^n)}\sum_{a\in\left(\mathbb{Z}/p^n\mathbb{Z}\right)^\times}  \\
\prod_{i=1}^k\left(\sum_{u_i\in H_{p^n}}\overline{\alpha_{p^n}(u_i;t_i)}\mathsf{Kl}_{p^n}(a-u_i,b_0)\right)^{m_i}\left(\sum_{v_i\in H_{p^n}}\alpha_{p^n}(v_i;t_i)\mathsf{Kl}_{p^n}(a-v_i,b_0)\right)^{n_i}
\end{multline*}
since the complete Kloosterman sums are real numbers. Expanding the powers, one gets
\begin{multline*}
\widetilde{\mathsf{M}_{p^n}}(\uple{t};\uple{m},\uple{n};b_0)=\frac{1}{p^{n\ell(\uple{m}+\uple{n})/2}\varphi(p^n)}\sum_{a\in\left(\mathbb{Z}/p^n\mathbb{Z}\right)^\times}\prod_{i=1}^k\sum_{\uple{u}_i=\left(u_{i,1},\dots,u_{i,m_i}\right)\in H_{p^n}^{m_i}}\sum_{\uple{v}_i=\left(v_{i,1},\dots,v_{i,n_i}\right)\in H_{p^n}^{n_i}} \\
\prod_{e_i=1}^{m_i}\overline{\alpha_{p^n}(u_{i,e_i};t_i)}\mathsf{Kl}_{p^n}(a-u_{i,e_i},b_0)\prod_{f_i=1}^{n_i}\alpha_{p^n}(v_{i,f_i};t_i)\mathsf{Kl}_{p^n}(a-v_{i,f_i},b_0).
\end{multline*}
Let us set for $1\leq i\leq k$,
\begin{equation*}
\uple{h}_i=\left(h_{i,1},\dots,h_{i,m_i},h_{i,m_i+1},\dots,h_{i,m_i+n_i}\right)=\left(u_{i,1},\dots,u_{i,m_i},v_{i,1},\dots,v_{i,n_i}\right)\in H_{p^n}^{m_i+n_i}
\end{equation*}
and
\begin{equation*}
\uple{h}=(\uple{h}_1,\dots,\uple{h}_k)\in H_{p^n}^{\ell(\uple{m}+\uple{n})}.
\end{equation*}
Exchanging the order of summations, one is led to
\begin{multline*}
\widetilde{\mathsf{M}_{p^n}}(\uple{t};\uple{m},\uple{n};b_0)=\frac{1}{p^{n\ell(\uple{m}+\uple{n})/2}}\sum_{\uple{h}\in H_{p^n}^{\ell(\uple{m}+\uple{n})}}\prod_{i=1}^k\prod_{j=1}^{m_i}\overline{\alpha_{p^n}(h_{i,j};t_i)}\prod_{j=m_i+1}^{m_i+n_i}\alpha_{p^n}(h_{i,j};t_i) \\
\frac{1}{\varphi(p^n)}\sum_{a\in\left(\mathbb{Z}/p^n\mathbb{Z}\right)^\times}\prod_{i=1}^k\prod_{j=1}^{m_i+n_i}\mathsf{Kl}_{p^n}(a-h_{i,j},b_0).
\end{multline*}
\par
By Lemma \ref{lemma_noncoprim} and \eqref{eq_bound_alpha}, the contribution of the tuples $\uple{h}$ different from the tuple $\uple{0}$ and whose components are not distinct modulo $p$ is bounded by
\begin{equation}\label{eq_not_coprime_p}
\ll_{\ell(\uple{m}+\uple{n})}\frac{\log^{\ell(\uple{m}+\uple{n})}\left(p^n\right)}{p}
\end{equation} 
where the implied constant only depends on $\ell(\uple{m}+\uple{n})$.
\par
Thus, up to all the previous error terms,
\begin{multline}\label{eq_caca}
\widetilde{\mathsf{M}_{p^n}}(\uple{t};\uple{m},\uple{n};b_0)=\frac{1}{p^{n\ell(\uple{m}+\uple{n})/2}}\sum_{\uple{h}\in H_{p^n}^{\ell(\uple{m}+\uple{n})}}^\ast\prod_{i=1}^k\prod_{j=1}^{m_i}\overline{\alpha_{p^n}(h_{i,j};t_i)}\prod_{j=m_i+1}^{m_i+n_i}\alpha_{p^n}(h_{i,j};t_i) \\
\frac{1}{\varphi(p^n)}\sum_{a\in\left(\mathbb{Z}/p^n\mathbb{Z}\right)^\times}\prod_{i=1}^k\prod_{j=1}^{m_i+n_i}\mathsf{Kl}_{p^n}(a-h_{i,j},b_0)
\end{multline}
where the $\ast$ means that the summation is over the tuples $\uple{h}=\left(h_{i,j}\right)_{\substack{1\leq i\leq k \\
1\leq j\leq m_i}}$ whose components are either equal or distinct modulo $p$, namely
\begin{equation*}
h_{i,j}=h_{k,\ell}\quad\text{ or }\quad p\nmid h_{i,j}-h_{k,\ell}
\end{equation*}
for any $(i,j)\neq(k,\ell)$ in the relevant ranges.
\par
Note that by \eqref{eq_Spmu},
\begin{equation*}
\frac{1}{\varphi(p^n)}\sum_{a\in\left(\mathbb{Z}/p^n\mathbb{Z}\right)^\times}\prod_{i=1}^k\prod_{j=1}^{m_i+n_i}\mathsf{Kl}_{p^n}(a-h_{i,j},b_0)=\mathsf{S}_{p^n}(\uple{\mu_{\uple{h}}};b_0)\\ 
\end{equation*}
where $\uple{\mu_{\uple{h}}}=\left(\mu_{\uple{h}}(\tau)\right)_{\tau\in\mathbb{Z}/p^n\mathbb{Z}}$ is the $p^n$-tuple of non-negative integers defined by
\begin{equation*}
\forall\tau\in\mathbb{Z}/p^n\mathbb{Z},\quad\mu_{\uple{h}}(\tau)\coloneqq\sum_{i=1}^k\left\vert\left\{j\in\{1,\dots,m_i+n_i\}, -h_{i,j}\equiv\tau\bmod{p^n}\right\}\right\vert.
\end{equation*}
Note also that
\begin{equation*}
\sum_{\tau\in\Z/p^n\Z}\mu_{\uple{h}}(\tau)=\ell(\uple{m}+\uple{n})
\end{equation*}
so that 
\begin{equation*}
\left\vert\left\{\tau\bmod{p}, \tau\in\Z/p^n\Z, \mu_{\uple{h}}(\tau)\geq 1\right\}\right\vert=\left\vert\left\{\tau\in\Z/p^n\Z, \mu_{\uple{h}}(\tau)\geq 1\right\}\right\vert\leq\ell(\uple{m}+\uple{n})
\end{equation*}
according to the property satisfied by the relevant tuples $\uple{h}$ in \eqref{eq_caca}.
\par
Hence, one can apply Proposition \ref{propo_shifted_Kl}. By \eqref{eq_bound_alpha}, the contribution of the error term is bounded by
\begin{equation*}
\ll_{\ell(\uple{m}+\uple{n}),\epsilon}\log^{\ell(\uple{m}+\uple{n})}\left(p^n\right)p^{-\frac{4(n-1)}{2^n}+\epsilon}
\end{equation*}
for any $\epsilon>0$, where the implied constant only depends on $\ell(\uple{m}+\uple{n})$ and $\epsilon$. Thus, up to the previous error terms,
\begin{multline*}
\widetilde{\mathsf{M}_{p^n}}(\uple{t};\uple{m},\uple{n};b_0)=\frac{1}{p^{n\ell(\uple{m}+\uple{n})/2}}\sum_{\uple{h}\in H_{p^n}^{\ell(\uple{m}+\uple{n})}}\prod_{i=1}^k\prod_{j=1}^{m_i}\overline{\alpha_{p^n}(h_{i,j};t_i)}\prod_{j=m_i+1}^{m_i+n_i}\alpha_{p^n}(h_{i,j};t_i) \\
\left[\prod_{\tau\in \mathbb{Z}/p^n\mathbb{Z}}\delta_{2\mid\mu_{\uple{h}}(\tau)}\binom{\mu_{\uple{h}}(\tau)}{\mu_{\uple{h}}(\tau)/2}\right]\frac{\left\vert\mathsf{A}_{p^n}(\mu_{\uple{h}}(\tau))\right\vert}{\varphi(p^n)}
\end{multline*}
where $\mathsf{A}_{p^n}(\mu_{\uple{h}}(\tau))$ is defined in \eqref{eq_apmu}.
\par
Let us apply Proposition \ref{propo_cardinality}. By \eqref{eq_bound_alpha}, the contribution of the error term is bounded by
\begin{equation*}
\ll_{\ell(\uple{m}+\uple{n})}\frac{\log^{\ell(\uple{m}+\uple{n})}\left(p^n\right)}{\sqrt{p}}
\end{equation*}
where the implied constant only depends on $\ell(\uple{m}+\uple{n})$ and, up to all the previous error terms,
\begin{multline*}
\widetilde{\mathsf{M}_{p^n}}(\uple{t};\uple{m},\uple{n};b_0)=\frac{1}{p^{n\ell(\uple{m}+\uple{n})/2}}\sum_{\uple{h}\in H_{p^n}^{\ell(\uple{m}+\uple{n})}}\prod_{i=1}^k\prod_{j=1}^{m_i}\overline{\alpha_{p^n}(h_{i,j};t_i)}\prod_{j=m_i+1}^{m_i+n_i}\alpha_{p^n}(h_{i,j};t_i) \\
\left[\prod_{\tau\in \mathbb{Z}/p^n\mathbb{Z}}\delta_{2\mid\mu_{\uple{h}}(\tau)}\binom{\mu_{\uple{h}}(\tau)}{\mu_{\uple{h}}(\tau)/2}\right]\frac{1}{2^{\abs{\mathsf{T}(\uple{\mu_{\uple{h}}})}}}.
\end{multline*}
It should be pointed out that the fact that
\begin{equation*}
\left\vert\mathsf{T}(\uple{\mu_{\uple{h}}})\right\vert=\left\vert\overline{\mathsf{T}}(\uple{\mu_{\uple{h}}})\right\vert
\end{equation*}
is crucial since recognizing the moments of the measure $\mu$ requires the left-hand side of the previous equation whereas the right-hand side appears by Proposition \ref{propo_cardinality}.
\par
By \eqref{eq_mu_1},
\begin{multline*}
\widetilde{\mathsf{M}_{p^n}}(\uple{t};\uple{m},\uple{n};b_0)=\frac{1}{p^{n\ell(\uple{m}+\uple{n})/2}}\sum_{\uple{h}\in H_{p^n}^{\ell(\uple{m}+\uple{n})}}^{\ast}\prod_{i=1}^k\prod_{j=1}^{m_i}\overline{\alpha_{p^n}(h_{i,j};t_i)}\prod_{j=m_i+1}^{m_i+n_i}\alpha_{p^n}(h_{i,j};t_i) \\
\mathbb{E}\left(\prod_{i=1}^{k}\prod_{j=1}^{m_i+n_i}U_{h_{i,j}}\right)
\end{multline*}
for any finite sequence of real-valued independent random variables $\left(U_h\right)_{h\in\Z/p^n\Z}$ of law the probability measure $\mu$ defined in \eqref{eq_def_mu_1}.  The fact that
\begin{equation*}
\mathbb{E}\left(\prod_{i=1}^{k}\prod_{j=1}^{m_i+n_i}U_{h_{i,j}}\right)=\mathbb{E}\left(\prod_{\tau\in\mathbb{Z}/p^n\mathbb{Z}}U_{\tau}^{\mu(\tau)}\right)
\end{equation*}
has also been used. One can add the missing tuples $\uple{h}$ at the admissible cost given in \eqref{eq_not_coprime_p} so that, up to all the previous error terms,
\begin{multline*}
\widetilde{\mathsf{M}_{p^n}}(\uple{t};\uple{m},\uple{n};b_0)=\frac{1}{p^{n\ell(\uple{m}+\uple{n})/2}}\sum_{\uple{h}\in H_{p^n}^{\ell(\uple{m}+\uple{n})}}\prod_{i=1}^k\prod_{j=1}^{m_i}\overline{\alpha_{p^n}(h_{i,j};t_i)}\prod_{j=m_i+1}^{m_i+n_i}\alpha_{p^n}(h_{i,j};t_i) \\
\mathbb{E}\left(\prod_{i=1}^{k}\prod_{j=1}^{m_i+n_i}U_{h_{i,j}}\right).
\end{multline*}
\par
Let us approximate the coefficients $\alpha_{p^n}(h;t)$. By \eqref{eq_approx_alpha} and \eqref{eq_bound_alpha}, one gets, up to all the previous error terms,
\begin{multline*}
\widetilde{\mathsf{M}_{p^n}}(\uple{t};\uple{m},\uple{n};b_0)=\frac{1}{p^{n\ell(\uple{m}+\uple{n})/2}}\sum_{\uple{h}\in H_{p^n}^{\ell(\uple{m}+\uple{n})}}\prod_{i=1}^k\prod_{j=1}^{m_i}\overline{\beta(h_{i,j};t_i)}\prod_{j=m_i+1}^{m_i+n_i}\beta(h_{i,j};t_i) \\
\mathbb{E}\left(\prod_{i=1}^{k}\prod_{j=1}^{m_i+n_i}U_{h_{i,j}}\right)+O_{\ell(\uple{m}+\uple{n})}\left(\frac{\log^{\ell(\uple{m}+\uple{n})-1}\left(p^n\right)}{p^n}\right).
\end{multline*}
Reverting the computation done at the very beginning of the proof of this proposition, one is led to
\begin{equation*}
\widetilde{\mathsf{M}_{p^n}}(\uple{t};\uple{m},\uple{n};b_0)=\mathbb{E}\left(\prod_{i=1}^{k}\overline{\mathsf{Kl}_{\frac{p^n-1}{2}}(t_i)}^{m_i}\mathsf{Kl}_{\frac{p^n-1}{2}}(t_i)^{n_i}\right)
\end{equation*}
up to all the previous error terms and where
\begin{equation*}
\mathsf{Kl}_{\frac{p^n-1}{2}}(t;\ast)=\sum_{\abs{h}\leq\frac{p^n-1}{2}}\beta(h;t)U_h(\ast).
\end{equation*}
\par
Finally, by \eqref{eq_supnorm} and \eqref{eq_L1} in Proposition \ref{propo_kahane}, up to all the previous error terms,
\begin{equation*}
\widetilde{\mathsf{M}_{p^n}}(\uple{t};\uple{m},\uple{n};b_0)=\mathbb{E}\left(\prod_{i=1}^{k}\overline{\mathsf{Kl}(t_i)}^{m_i}\mathsf{Kl}(t_i)^{n_i}\right)+O_{\ell(\uple{m}+\uple{n})}\left(\frac{\log^{\ell(\uple{m}+\uple{n})}\left(p^n\right)}{p^{n/2}}\right)
\end{equation*}
where
\begin{equation*}
\mathsf{Kl}(t;\ast)=\sum_{h\in\Z}\beta(h;t)U_h(\ast)
\end{equation*}
for any sequence of real-valued independent random variables $\left(U_h\right)_{h\in\Z}$ of law the probability measure $\mu$.
\end{proof}
\section{The tightness condition}\label{sec_tight}%
\subsection{The counting ingredient}%
The following lemma states, without any proof, the version of Hensel's lemma, which will be used in the proof of Lemma \ref{lemma_Hensel_2}. This result is so standard that we do not give any reference too.
\begin{lemma}[Hensel's lemma]\label{lemma_Hensel}
Let $k$ be a positive integer and $f$ be a polynomial with integer coefficients. Assume that $x_0$ is a solution modulo $p^k$ of the congruence $f(x)\equiv 0\bmod{p^k}$.
\begin{itemize}
\item
If $p\nmid f'(x_0)$ then there is exactly one solution modulo $p^{k+1}$ of the congruence $f(x)\equiv 0\bmod{p^{k+1}}$ congruent to $x_0$ modulo $p^k$.
\item
If $p\mid f'(x_0)$ and $f(x_0)\equiv 0\bmod{p^{k+1}}$ then there are exactly $p$ solutions modulo $p^{k+1}$ of the congruence $f(x)\equiv 0\bmod{p^{k+1}}$ congruent to $x_0$ modulo $p^k$. They are given by $x_0+p^kj$ for $j$ modulo $p$.
\end{itemize} 
\end{lemma}
\begin{lemma}[Hensel's lemma in degree $2$]\label{lemma_Hensel_2}
Let $n\geqslant 1$ be an integer and $f(X)=X^2-sX+\pi$ be a polynomial of degree $2$ with integer coefficients satisfying $s\equiv\pi+1\bmod{p^n}$. Assume that $p^\ell\mid\mid\pi-1$ for some integer $\ell\geq 1$. The number of solutions of the congruence $f(x)\equiv 0\bmod{p^n}$ equals
\begin{equation*}
\begin{cases}
2p^\ell & \text{if $\;1\leq\ell\leq n/2-1$,} \\
p^{n/2} & \text{if $\;n/2\leq\ell\leq n$,}
\end{cases}
\end{equation*}
if $n$ is even, and
\begin{equation*}
\begin{cases}
2p^\ell & \text{if $\;1\leq\ell\leq (n-1)/2$,} \\
p^{(n-1)/2} & \text{if $\;(n-1)/2+1\leq\ell\leq n$,}
\end{cases}
\end{equation*}
if $n$ is odd.
\end{lemma}
\begin{remark}
This lemma is proved by a quite technical induction on $n\geq 1$ but understanding the set of solutions for $1\leq n\leq 3$ of $f(x)\equiv 0\bmod{p^n}$ gives an idea of how the induction works.
\end{remark}
\begin{proof}[\proofname{} of lemma \ref{lemma_Hensel_2}]%
In this proof, recall that $p\mid\pi-1$.
\par
Obviously, $1$ is the only solution of the congruence $f(x)\equiv 0\bmod{p}$ with $s\equiv\pi+1\bmod{p}$.
\par
Let us quickly check what happens for $n=2$. One has $f(1)\equiv 0\bmod{p^2}$ and $f^\prime(1)=2-s\equiv 0\bmod{p}$. By Lemma \ref{lemma_Hensel}, the only solutions of $f(x)\equiv 0\bmod{p^2}$ with $s\equiv\pi+1\bmod{p^2}$ are $1+pk_1$ for $0\leq k_1<p$.
\par
Let us do the case $n=3$. For $0\leq k_1<p$, $1+pk_1$ is a solution of $f(x)\equiv 0\bmod{p^2}$ satisfying $f^\prime(1+pk_1)\equiv 0\bmod{p}$ and
\begin{equation*}
f\left(1+pk_1\right)\equiv pk_1(1-\pi+pk_1)\bmod{p^3}.
\end{equation*}
If $k_1=0$ then by Lemma \ref{lemma_Hensel}, $1+p^2k_2$ for $0\leq k_2<p$ are the only solutions of $f(x)\equiv 0\bmod{p^3}$ congruent to $1$ modulo $p^2$. Otherwise, $p\mid\mid\pi-1$ and $k_{1,\pi}$ must be the unique invertible integer modulo $p$ satisfying $pk_{1,\pi}\equiv\pi-1\bmod{p^2}$. Then, by Lemma \ref{lemma_Hensel} $1+pk_{1,\pi}+p^2k_2$ for $0\leq k_2<p$ are the solutions of $f(x)\equiv 0\bmod{p^3}$ congruent to $1+pk_{1,\pi}$ modulo $p^2$. We have just seen that the solutions of $f(x)\equiv 0\bmod{p^3}$ with $s\equiv\pi+1\bmod{p^3}$ are
\begin{itemize}
\item
$1+p^2k_2$ for $0\leq k_2<p$,
\item
$1+pk_{1,\pi}+p^2k_2$ for $0\leq k_2<p$ and if $p\mid\mid\pi-1$ and $pk_{1,\pi}\equiv\pi-1\bmod{p^2}$. 
\end{itemize}
\par
Note that the previous simple use of Hensel's lemma proves Lemma \ref{lemma_Hensel_2} for $1\leq n\leq 3$. We will conclude by induction on $n\geq 2$. 
\par
Let us set
\begin{equation*}
\left(\ell_1(n),\ell_2(n)\right)\coloneqq\begin{cases}
\left(n/2-1,n/2\right) & \text{if $n$ is even,} \\
\left((n-1)/2,(n-1)/2\right) & \text{if $n$ is odd.}
\end{cases}
\end{equation*}
Let us prove that for any $n\geq 2$, the solutions of the congruence $f(x)\equiv 0\bmod{p^n}$ for any polynomial $f(X)=X^2-sX+\pi$ of degree $2$ with integer coefficients satisfying $s\equiv\pi+1\bmod{p^n}$ are
\begin{equation*}
1+p^{n-1}k_{n-1}
\end{equation*}
where $0\leq k_{n-1}<p$ and for any $2\leq m\leq\ell_2(n)$,
\begin{equation*}
1+p^{n-m}k_{n-m}+\dots+p^{n-1}k_{n-1}
\end{equation*}
where $0<k_{n-m}<p$, $0\leq k_{n-m+1},\dots,k_{n-1}<p$ provided that $p^m\mid\pi-1$ and for any $2\leq m\leq\ell_1(n)$,
\begin{equation*}
1+p^mk_{m,\pi}+\dots+p^{n-m-1}k_{n-m-1,\pi}+p^{n-m}k_{n-m}+\dots+p^{n-1}k_{n-1}
\end{equation*}
where $0\leq k_{n-m},\dots,k_{n-1}<p$ provided that $p^m\mid\mid\pi-1$. Here, the numbers $k_{u,\pi}$, $m\leq u\leq n-m-1$,  are fixed integers modulo $p$ satisfying
\begin{equation*}
p^m k_{m,\pi}+\dots+p^{n-m-1}k_{n-m-1}\equiv\pi-1\bmod{p^{n-m}}.
\end{equation*}
In particular, $k_{m,\pi}$ is invertible modulo $p$. This fact trivially implies Lemma \ref{lemma_Hensel_2} for $n\geq 2$. The cases $n=1$, $n=2$ and $n=3$ have just been seen above.
\par
Let $n\geq 2$. Let us assume that the result holds at the rank $n$ and let us check that it remains true at the rank $n+1$. For instance, let us assume that $n$ is even. We do not provide the proof when $n$ is odd since this is completely similar.
\par
For $0\leq k_{n-1}<p$, $x_n\coloneqq 1+p^{n-1}k_{n-1}$ is a solution of $f(x)\equiv 0\bmod{p^n}$, which satisfies  $f^\prime(x_n)\equiv 0\bmod{p}$ and
\begin{equation*}
f(x_n)\equiv p^{n-1}k_{n-1}\left(1-\pi+p^{n-1}k_{n-1}\right)\bmod{p^{n+1}}.
\end{equation*}
By Lemma \ref{lemma_Hensel}, the only solutions of $f(x)\equiv 0\bmod{p^{n+1}}$ congruent to $x_n$ are
\begin{itemize}
\item
$1+p^nk_n$ for $0\leq k_n<p$ if $k_{n-1}=0$,
\item
$1+p^{n-1}k_{n-1}+p^nk_n$ for $0<k_{n-1}<p$, $0\leq k_n<p$ and if $p^2\mid\pi-1$.
\end{itemize}
\par
Let $2\leq m\leq\ell_2(n)=n/2$. Assume that $p^m\mid\pi-1$. For $0\leq k_{n-m+1},\dots,k_{n-1}<p$ and $0<k_{n-m}<p$, $x_{m,n}\coloneqq 1+p^{n-m}k_{n-m}+\dots+p^{n-1}k_{n-1}$ is a solution $f(x)\equiv 0\bmod{p^n}$, which satisfies  $f^\prime(x_{m,n})\equiv 0\bmod{p}$ and
\begin{multline*}
f(x_{m,n})\equiv p^{n}\left(k_{n-m}+\dots+p^{m-1}k_{n-1}\right)\left(\frac{1-\pi}{p^m}+p^{n-2m}k_{n-m}+\dots+p^{n-1-m}k_{n-1}\right) \\
\bmod{p^{n+1}}.
\end{multline*}
If $2\leq m<n/2$ then by Lemma \ref{lemma_Hensel}, the only solutions of $f(x)\equiv 0\bmod{p^{n+1}}$ congruent to $x_{m,n}$ are
\begin{equation*}
1+p^{n-m}k_{n-m}+\dots+p^{n-1}k_{n-1}+p^nk_n
\end{equation*}
where $0<k_{n-m}<p$, $0\leq k_{n-m+1},\dots,k_{n}<p$ provided that $p^{m+1}\mid\pi-1$. If $m=n/2$ then by Lemma \ref{lemma_Hensel}, the only solutions of $f(x)\equiv 0\bmod{p^{n+1}}$ congruent to $x_{n/2,n}$ are
\begin{equation*}
1+p^{n/2}k_{n/2,\pi}+p^{n/2+1}k_{n/2+1}+\dots+p^{n-1}k_{n-1}+p^nk_n
\end{equation*}
where $0\leq k_{n/2+1},\dots,k_{n}<p$ provided that $p^{n/2}\mid\mid\pi-1$ and $p^{n/2}k_{n/2,\pi}\equiv\pi-1\bmod{p^{n/2+1}}$.
\par
Let $1\leq m\leq\ell_1(n)=n/2-1$. Assume that $p^m\mid\mid\pi-1$. For $0\leq k_{n-m},\dots,k_{n-1}<p$, 
\begin{equation*}
x_{m,n}\coloneqq 1+p^mk_{m,\pi}+\dots+p^{n-m-1}k_{n-m-1,\pi}+p^{n-m}k_{n-m}+\dots+p^{n-1}k_{n-1}
\end{equation*}
is a solution $f(x)\equiv 0\bmod{p^n}$, which satisfies  $f^\prime(x_{m,n})\equiv 0\bmod{p}$ and
\begin{multline*}
f(x_{m,n})\equiv p^{n}\left(k_{m,\pi}+\dots+p^{n-2m-1}k_{n-m-1,\pi}+p^{n-2m}k_{n-m}+\dots+p^{n-1-m}k_{n-1}\right) \\
\left(\frac{1-\pi+p^{m}k_{m,\pi}+\dots+p^{n-1-m}k_{n-m-1,\pi}}{p^{n-m}}+k_{n-m}+pk_{n-m+1}+\dots+k_{n-1}p^{n-1-m}\right) \\
\bmod{p^{n+1}}.
\end{multline*}
By Lemma \ref{lemma_Hensel}, the only solutions of $f(x)\equiv 0\bmod{p^{n+1}}$ congruent to $x_{m,n}$ are
\begin{equation*}
1+p^mk_{m,\pi}+\dots+p^{n-m}k_{n-m,\pi}+p^{n-m+1}k_{n-m+1}+\dots+p^nk_n
\end{equation*}
where $0\leq k_{n-m+1},\dots,k_{n}<p$ and where
\begin{equation*}
p^m k_{m,\pi}+\dots+p^{n-m}k_{n-m,\pi}\equiv\pi-1\bmod{p^{n-m+1}}.
\end{equation*}
This completes the induction on $n$.
\end{proof}
\begin{proposition}[The counting ingredient]\label{propo_counting}
Let $n\geq 1$ be an integer and $I$ be a non-empty interval in $\left(\Z/p^n\Z\right)^\times$. The number of quadruples $(x_1,x_2,x_3,x_4)\in I^4$ satisfying
\begin{eqnarray*}
x_1+x_2 & \equiv & x_3+x_4\bmod{p^n}\\
\overline{x_1}+\overline{x_2} & \equiv & \overline{x_3}+\overline{x_4}\bmod{p^n}
\end{eqnarray*}
is bounded by an absolute positive constant times $n\abs{I}^2$.
\end{proposition}
\begin{proof}[\proofname{} of proposition \ref{propo_counting}]%
Let us denote by $N_{k,\ell}(p^n;I)$ the number of these quadruples $(x_1,x_2,x_3,x_4)$ satisfying $p^k\mid\mid x_3+x_4$ and $p^\ell\mid\mid x_3-x_4$ for some fixed integers $k,\ell\in\{0,\dots,n\}$, which must satisfy $k\ell=0$ since $p$ is odd. Let $(x_1,x_2,x_3,x_4)$ be such a quadruple. Let us fix $x_4$. There are at most $\abs{I}$ such $x_4$. The bijective change of variables $y_i=\overline{x_4}x_i$ for $1\leq i\leq 3$ leads to the system
\begin{eqnarray*}
y_1+y_2 & \equiv & y_3+1\bmod{p^n} \\
\overline{y_1}+\overline{y_2} & \equiv & \overline{y_3}+1\bmod{p^n}
\end{eqnarray*}
where the triple $(y_1,y_2,y_3)$ belongs to $\left(\overline{x_4}I\right)^3$ and whose components satisfy $p^k\mid\mid y_3+1$ and $p^\ell\mid\mid y_3-1$. Let us set
\begin{equation*}
s=y_1+y_2\quad\text{ and }\quad\varpi=y_1 y_2.
\end{equation*}
The previous system becomes
\begin{equation}\label{eq_s}
s\equiv y_3+1\bmod{p^n}
\end{equation}
and
\begin{equation*}
\overline{\varpi}s\equiv\overline{y_3}+1\bmod{p^n}.
\end{equation*}
Thus,
\begin{equation*}
1\equiv y_3\overline{y_3}\equiv(s-1)(\overline{\varpi}s-1)\bmod{p^n}
\end{equation*}
so that $s(s-(\varpi+1))\equiv 0\bmod{p^n}$ and
\begin{equation}\label{eq_s_pi}
s\equiv\varpi+1\bmod{p^{n-k}}.
\end{equation}
Let $f(X)=X^2-sX+\varpi$. Obviously, $y_1$ and $y_2$ are solutions modulo $p^n$ of the congruence
\begin{equation}\label{eq_pol_congr_n}
f(x)\equiv 0\bmod{p^{n}}. 
\end{equation}
Note also that
\begin{equation}\label{eq_pol_congr_n-k}
f(X)\equiv(X-1)(X-\varpi)\bmod{p^{n-k}}
\end{equation}
by \eqref{eq_s_pi}. In particular, if $n-k\geq 1$ then the only solutions modulo $p$ of
\begin{equation*}
f(x)\equiv 0\bmod{p}
\end{equation*}
are $1$ and $\varpi$, which satisfy $f^\prime(\varpi)\equiv-f^\prime(1)=\varpi-1\mod{p}$ by \eqref{eq_s_pi}. Let us consider three distinct cases.
\par
\underline{First case}: $k=0$ and $0\leq\ell\leq n$. In this case, $p\nmid s$ by \eqref{eq_s} and $p^\ell\mid\mid\varpi-1$ since $y_3-1\equiv s-2\equiv\varpi-1\bmod{p^n}$ by \eqref{eq_s} and \eqref{eq_s_pi}. Let us fix $y_3$, which implies that $s$ is fixed by \eqref{eq_s} and $\varpi$ is fixed by \eqref{eq_s_pi}. There are $\ll 1+\abs{I}/p^\ell$ such $y_3$. By Lemma \ref{lemma_Hensel_2}, the number $N_\ell(p^n)$ of solutions modulo $p^n$ of the congruence \eqref{eq_pol_congr_n} satisfies
\begin{equation*}
N_\ell(p^n)\ll\begin{cases}
p^\ell & \text{if $n$ is even and $0\leq\ell\leq n/2-1$,} \\
p^{n/2} & \text{if $n$ is even and $n/2\leq\ell\leq n$,} \\
p^\ell & \text{if $n$ is odd and $0\leq\ell\leq (n-1)/2$,} \\
p^{(n-1)/2} & \text{if $n$ is odd and $(n-1)/2+1\leq\ell\leq n$.}
\end{cases}
\end{equation*}
In total,
\begin{equation*}
N_{0,\ell}(p^n;I)\ll\abs{I}\left(1+\frac{\abs{I}}{p^\ell}\right)\min{\left(N_\ell(p^n),\abs{I}\right)}.
\end{equation*}
\par
\underline{Second case}: $1\leq k\leq n-1$ and $\ell=0$. In this case, $p^k\mid\mid s$ by \eqref{eq_s} and $p\nmid\varpi-1$ by \eqref{eq_s} and \eqref{eq_s_pi}. Let us fix $y_3$, which implies that $s$ is fixed by \eqref{eq_s} and $\varpi$ is fixed modulo $p^{n-k}$ by \eqref{eq_s_pi}. There are $\ll 1+\abs{I}/p^k$ such $y_3$. By Lemma \ref{lemma_Hensel_2}, the congruence $f(x)\equiv 0\bmod{p^{n-k}}$ has exactly two solutions. Hence, the same holds for the number of pairs $(y_1,y_2)$ modulo $p^{n-k}$. In total,
\begin{equation*}
N_{k,0}(p^n;I)\ll\abs{I}\left(1+\frac{\abs{I}}{p^k}\right)\left(1+\frac{\abs{I}}{p^{n-k}}\right)\ll\abs{I}^2.
\end{equation*}
\par 
\underline{Third case}: $k=n$ and $\ell=0$. In this case $y_3\equiv-1\bmod{p^n}$ is fixed and given $y_1$, $y_2$ is fixed. In total,
\begin{equation*}
N_{n,0}(p^n;I)\ll\abs{I}^2.
\end{equation*}
\par
Altogether, the number of quadruples $(x_1,x_2,x_3,x_4)$ equals
\begin{equation*}
\sum_{\ell=0}^{n}N_{0,\ell}(p^n;I)+\sum_{k=1}^{n-1}N_{k,0}(p^n;I)+N_{n,0}(p^n;I)
\end{equation*}
and is bounded by $\ll n\abs{I}^2$.
\end{proof}
\subsection{The fourth moment of incomplete Kloosterman sums}%
\begin{proposition}[Bounding the fourth moment]\label{propo_4_moment}
Let $n\geqslant 2$ be an integer and $I$ be a non-empty interval in $\left(\Z/p^n\Z\right)^\times$. One has
\begin{equation*}
\mathsf{M}_4(I)\coloneqq\frac{1}{\varphi\left(p^n\right)^2}\sum_{(a,b)\in\left(\Z/p^n\Z\right)^\times\times\left(\Z/p^n\Z\right)^\times}\left\vert\frac{1}{p^{n/2}}\sum_{x\in I}e_{p^n}\left(ax+b\overline{x}\right)\right\vert^4\ll\frac{n\abs{I}^2}{\varphi\left(p^n\right)^2}.
\end{equation*}
\end{proposition}
\begin{proof}[\proofname{} of proposition \ref{propo_4_moment}]%
Expanding the fourth power, one is led to
\begin{multline*}
\mathsf{M}_4(I)=\frac{1}{\varphi\left(p^n\right)^2p^{2n}}\sum_{(x_1,x_2,x_3,x_4)\in I^4}\left(\sum_{a\in\left(\Z/p^n\Z\right)^\times}e_{p^n}\left(a(x_1+x_2-x_3-x_4)\right)\right) \\
\left(\sum_{b\in\left(\Z/p^n\Z\right)^\times}e_{p^n}\left(b(\overline{x_1}+\overline{x_2}-\overline{x_3}-\overline{x_4})\right)\right).
\end{multline*}
The orthogonality of additive characters ensures that
\begin{equation*}
\sum_{c\in\left(\Z/p^n\Z\right)^\times}e_{p^n}\left(cz\right)=p^n\delta_{z\equiv 0\bmod{p^n}}-p^{n-1}\delta_{z\equiv 0\bmod{p^{n-1}}}
\end{equation*}
for any $z$ in $\Z/p^n\Z$. Thus,
\begin{multline*}
\mathsf{M}_4(I)=\frac{1}{\varphi\left(p^n\right)^2}\sum_{\substack{(x_1,x_2,x_3,x_4)\in I^4 \\
x_1+x_2\equiv x_3+x_4\bmod{p^n} \\
\overline{x_1}+\overline{x_2}\equiv\overline{x_3}+\overline{x_4}\bmod{p^n}}}1-\frac{1}{\varphi\left(p^n\right)^2p}\sum_{\substack{(x_1,x_2,x_3,x_4)\in I^4 \\
x_1+x_2\equiv x_3+x_4\bmod{p^n} \\
\overline{x_1}+\overline{x_2}\equiv\overline{x_3}+\overline{x_4}\bmod{p^{n-1}}}}1 \\
-\frac{1}{\varphi\left(p^n\right)^2p}\sum_{\substack{(x_1,x_2,x_3,x_4)\in I^4 \\
x_1+x_2\equiv x_3+x_4\bmod{p^{n-1}} \\
\overline{x_1}+\overline{x_2}\equiv\overline{x_3}+\overline{x_4}\bmod{p^{n}}}}1+\frac{1}{\varphi\left(p^n\right)^2p^2}\sum_{\substack{(x_1,x_2,x_3,x_4)\in I^4 \\
x_1+x_2\equiv x_3+x_4\bmod{p^{n-1}} \\
\overline{x_1}+\overline{x_2}\equiv\overline{x_3}+\overline{x_4}\bmod{p^{n-1}}}}1.
\end{multline*}
Proposition \ref{propo_counting} completes the proof.
\end{proof}
\subsection{The tightness condition via Kolmogorov's criterion}%
\begin{proposition}[Tightness]\label{propo_tightne}
Let $n\geq 2$ be an integer. The sequence of $C^0([0,1],\C)$-valued random variables $\mathsf{Kl}_{p^n}(\ast;(\ast,\ast))$ on the random space $\left(\Z/p^n\Z\right)^\times\times\left(\Z/p^n\Z\right)^\times$ where $p$ is an odd prime number is tight.
\end{proposition}
\begin{proof}[\proofname{} of proposition \ref{propo_tightne}]%
Let us show that if $0\leq s,t\leq 1$ then
\begin{equation}\label{eq_kolmo}
\frac{1}{\varphi\left(p^n\right)^2}\sum_{(a,b)\in\left(\Z/p^n\Z\right)^\times\times\left(\Z/p^n\Z\right)^\times}\left\vert\mathsf{Kl}_{p^n}(t;(a,b))-\mathsf{Kl}_{p^n}(s;(a,b))\right\vert^4\ll n\abs{t-s}^2
\end{equation}
where the implied constant is absolute. The bound \eqref{eq_kolmo} is enough by Proposition \ref{propo_kolmo_tight} to ensure the tightness of the sequence of $C^0([0,1],\C)$-valued random variables $\mathsf{Kl}_{p^n}(\ast;(\ast,\ast))$ as $p$ tends to infinity among the odd prime numbers. One can assume that $0\leq s<t\leq 1$.
\par
\underline{First range}: $0\leq t-s\leq1\left/\left(\varphi\left(p^n\right)-1\right)\right.$ so that
\begin{equation}\label{eq_tec_3}
p^n\leq\frac{4}{t-s}.
\end{equation}
Let us show that
\begin{equation*}
\left\vert\mathsf{Kl}_{p^n}(t;(a,b))-\mathsf{Kl}_{p^n}(s;(a,b))\right\vert\leq 2\sqrt{t-s},
\end{equation*}
which implies \eqref{eq_kolmo} in this range. Let us assume that
\begin{equation*}
\frac{j-1}{\varphi\left(p^n\right)-1}\leq t\leq\frac{j}{\varphi\left(p^n\right)-1}
\end{equation*}
where $1\leq j\leq\varphi\left(p^n\right)-1$. Two cases can occur.
\par
\noindent{\underline{First case}}:
\begin{equation*}
\frac{j-1}{\varphi\left(p^n\right)-1}\leq s<t\leq\frac{j}{\varphi\left(p^n\right)-1}.
\end{equation*}
In this case,
\begin{multline*}
\left\vert\mathsf{Kl}_{p^n}(t;(a,b))-\mathsf{Kl}_{p^n}(s;(a,b))\right\vert=\left\vert\alpha_j((a,b);p^n)\right\vert(t-s) \leq\frac{\varphi\left(p^n\right)-1}{p^{n/2}}(t-s) \\
\leq 2\sqrt{t-s}
\end{multline*}
by \eqref{eq_tec_0} and \eqref{eq_tec_3}.
\par
\noindent{\underline{Second case}}:
\begin{equation*}
\frac{j-2}{\varphi\left(p^n\right)-1}\leq s\leq\frac{j-1}{\varphi\left(p^n\right)-1}\leq t\leq\frac{j}{\varphi\left(p^n\right)-1}
\end{equation*}
where $2\leq j\leq\varphi\left(p^n\right)-1$. In this case,
\begin{multline*}
\left\vert\mathsf{Kl}_{p^n}(t;(a,b))-\mathsf{Kl}_{p^n}(s;(a,b))\right\vert\leq\left\vert\mathsf{Kl}_{p^n}(t;(a,b))-z_j((a,b);p^n)\right\vert \\
+\left\vert z_j((a,b);p^n)-\mathsf{Kl}_{p^n}(s;(a,b))\right\vert.
\end{multline*}
The first term is less than
\begin{equation*}
\left\vert\alpha_j((a,b);p^n)\right\vert\left(t-\frac{j-1}{\varphi\left(p^n\right)-1}\right)
\end{equation*}
whereas the second term is less than
\begin{equation*}
\left\vert\alpha_{j-1}((a,b);p^n)\right\vert\left(\frac{j-1}{\varphi\left(p^n\right)-1}-s\right).
\end{equation*}
Altogether,
\begin{equation*}
\left\vert\mathsf{Kl}_{p^n}(t;(a,b))-\mathsf{Kl}_{p^n}(s;(a,b))\right\vert\leq\frac{\varphi\left(p^n\right)-1}{p^{n/2}}(t-s)\leq 2\sqrt{t-s}
\end{equation*}
by \eqref{eq_tec_0} and \eqref{eq_tec_3}. 
\par
\underline{Second range}: $t-s\geq 1\left/\left(\varphi\left(p^n\right)-1\right)\right.$ so that
\begin{equation}\label{eq_tec_4}
p^n\geq\frac{1}{t-s}.
\end{equation}
Let us assume that
\begin{equation*}
\frac{j-1}{\varphi\left(p^n\right)-1}<s\leq\frac{j}{\varphi\left(p^n\right)-1}\quad\text{ and }\quad\frac{k-1}{\varphi\left(p^n\right)-1}<t\leq\frac{k}{\varphi\left(p^n\right)-1}
\end{equation*}
where $1\leq j\leq k-1\leq \varphi\left(p^n\right)-2$. In other words,
\begin{equation*}
j=\left\lceil{(\varphi\left(p^n\right)-1)s}\right\rceil\quad\text{ and }\quad k=\left\lceil{(\varphi\left(p^n\right)-1)t}\right\rceil.
\end{equation*}
By \eqref{eq_bemol} and H\"{o}lder's inequality,
\begin{multline*}
\frac{1}{\varphi\left(p^n\right)^2}\sum_{(a,b)\in\left(\Z/p^n\Z\right)^\times\times\left(\Z/p^n\Z\right)^\times}\left\vert\mathsf{Kl}_{p^n}(t;(a,b))-\mathsf{Kl}_{p^n}(s;(a,b))\right\vert^4 \\
=\mathsf{M}_4(I_{s,t})+O\left(\frac{1}{p^{2n}}\right)=\mathsf{M}_4(I_{s,t})+O\left((t-s)^2\right)
\end{multline*}
where $I_{s,t}$ is the non-empty interval in $\left(\Z/p^n\Z\right)^\times$ given by
\begin{equation*}
\left(x_j(s)=\varphi\left(p^n\right)s+j-1\right.,\left.x_k(t)=\varphi\left(p^n\right)t+k-1\right]\cap\left(\Z/p^n\Z\right)^\times.
\end{equation*}
Its length satisfies
\begin{align*}
\abs{I_{s,t}} & =\left\lfloor{x_k(t)}\right\rfloor-\left\lceil{x_j(s)}\right\rceil \\
& \leq\varphi\left(p^n\right)(t-s)+\left\lceil{(\varphi\left(p^n\right)-1)t}\right\rceil-\left\lceil{(\varphi\left(p^n\right)-1)s}\right\rceil \\
& \leq 4(\varphi\left(p^n\right)-1)(t-s)+1 \\
& \leq 8(\varphi\left(p^n\right)-1)(t-s)
\end{align*}
since $(\varphi\left(p^n\right)-1)(t-s)\geq 1$. Proposition \ref{propo_4_moment} implies \eqref{eq_kolmo}.
\end{proof}
\section{Proof of Theorem \ref{theo_A} and Theorem \ref{theo_B}}\label{sec_final}%
Let us prove Theorem \ref{theo_A}. By Proposition \ref{propo_kahane}, the random variable $\mathsf{Kl}$ has moments to all orders. Thus, we are allowed to use the method of moments. Proposition \ref{propo_moment_tilde} leads to the result.
\par
Theorem \ref{theo_B} is implied by Theorem \ref{theo_prokho}, Theorem \ref{theo_A} and Proposition \ref{propo_tightne}.
\appendix
\section{Probabilistic tools}\label{sec_proba}%
This section contains some probabilistic results needed in this work. The main reference for both the statements and their proof is \cite{Ko}.
\par
Let us say a few words about random variables with values in the Banach space $C^0([0,1],\C)$ of $\C$-valued continuous function on $[0,1]$ endowed with the supremum norm. Confer \cite[Section B.9]{Ko} for more details. For each $n\geq 1$, let $X_n$ be a random variable on the random space $(\Omega_n,\mathcal{A}_n,\mathbb{P}_n)$ with values in $C^0([0,1],\C)$. Let $X$ be a $C^0([0,1],\C)$-valued random variable.
\par
The sequence $\left(X_n\right)_{n\geq 1}$ converges to $X$ in the sense of \emph{finite distributions} if for all integers $k\geq 1$ and all $k$-tuples $(t_1,\dots,t_k)$ with
\begin{equation*}
0\leq t_1<\dots<t_k\leq 1,
\end{equation*}
the sequence of $\C^k$-valued random vectors $(X_n(t_1),\dots,X_n(t_k))$ converge in law to the random vector $(X(t_1),\dots,X(t_k))$.
\par
The sequence $\left(X_n\right)_{n\geq 1}$ converges in \emph{law} to $X$ if for any $\C$-valued continuous and bounded map $\varphi$ on the Banach space $C^0([0,1],\C)$, the sequence of complex numbers $\left(\mathbb{E}\left(\varphi(X_n)\right)\right)_{n\geq 1}$ converges to $\mathbb{E}(\varphi(X))$.
\par
Each $X_n$ induces a probability measure $\mu_n$ on the Banach space by
\begin{equation*}
\forall A\subset C^0([0,1],\C),\quad\mu_n(A)=\mathbb{P}_n\left(X_n^{-1}(A)\right).
\end{equation*}
The sequence $\left(X_n\right)_{n\geq 1}$ is said to be \emph{tight} if for any $\epsilon>0$, there exists a compact subset $K$ of $C^0([0,1],\C)$ satisfying
\begin{equation*}
\forall n\geq 1,\quad\mu_n(K)\geq 1-\epsilon.
\end{equation*}
A practical criterion for tightness is due to Kolmogorov.
\begin{proposition}[Kolmogorov's criterion for tightness]\label{propo_kolmo_tight}
If there exists $\alpha>0$ and $\delta>0$ so that
\begin{equation*}
\forall(s,t)\in[0,1]^2,\quad\mathbb{E}\left(\left\vert X_n(s)-X_n(t)\right\vert^\alpha\right)\ll\abs{s-t}^{1+\delta}
\end{equation*}
then $\left(X_n\right)_{n\geq 1}$ is tight. 
\end{proposition}
\begin{remark}
This is \cite[Proposition B.9.5 Page 82]{Ko}.
\end{remark}
Last but not least, the main tool of this work is Prokhorov's criterion for convergence in law in $C^0([0,1],\C)$.
\begin{theorem}[Prokhorov's criterion]\label{theo_prokho}
If $\left(X_n\right)_{n\geq 1}$ converges to $X$ in the sense of finite distributions and $\left(X_n\right)_{n\geq 1}$ is tight then $\left(X_n\right)_{n\geq 1}$ converges in law in the sense of $C^0([0,1],\C)$-valued random variables.
\end{theorem}
\begin{remark}
This is \cite[Theorem B.9.4 Page 82]{Ko}.
\end{remark}
\bibliographystyle{alpha}
\bibliography{biblio}

\begin{thebibliography}{BGGS13}

\bibitem[BG13]{MR3081779}
Jonathan~W. Bober and Leo Goldmakher.
\newblock The distribution of the maximum of character sums.
\newblock {\em Mathematika}, 59(2):427--442, 2013.

\bibitem[BGGS13]{BoGoGaSo}
J.W. Bober, L.~Goldmakher, A.~Granville, and K.~Soundararajan.
\newblock The frequency and the structure of large character sums.
\newblock 2013.
\newblock preprint.

\bibitem[GS07]{MR2276774}
Andrew Granville and K.~Soundararajan.
\newblock Large character sums: pretentious characters and the
  {P}\'olya-{V}inogradov theorem.
\newblock {\em J. Amer. Math. Soc.}, 20(2):357--384, 2007.

\bibitem[GVW16]{GoVeWa}
Ke~Gong, Willem Veys, and Daqing Wan.
\newblock Power moments of {K}loosterman sums.
\newblock {\em J. Number Theory}, 164:103--126, 2016.

\bibitem[IK04]{IwKo}
Henryk Iwaniec and Emmanuel Kowalski.
\newblock {\em Analytic number theory}, volume~53 of {\em American Mathematical
  Society Colloquium Publications}.
\newblock American Mathematical Society, Providence, RI, 2004.

\bibitem[Kah85]{MR0254888}
Jean-Pierre Kahane.
\newblock {\em Some random series of functions}, volume~5 of {\em Cambridge
  Studies in Advanced Mathematics}.
\newblock Cambridge University Press, Cambridge, second edition, 1985.

\bibitem[Kat80]{MR617009}
Nicholas~M. Katz.
\newblock {\em Sommes exponentielles}, volume~79 of {\em Ast\'erisque}.
\newblock Soci\'et\'e Math\'ematique de France, Paris, 1980.
\newblock Course taught at the University of Paris, Orsay, Fall 1979, With a
  preface by Luc Illusie, Notes written by G\'erard Laumon, With an English
  summary.

\bibitem[Kel10]{MR2646758}
Dubi Kelmer.
\newblock Distribution of twisted {K}loosterman sums modulo prime powers.
\newblock {\em Int. J. Number Theory}, 6(2):271--280, 2010.

\bibitem[Kob84]{MR754003}
Neal Koblitz.
\newblock {\em {$p$}-adic numbers, {$p$}-adic analysis, and zeta-functions},
  volume~58 of {\em Graduate Texts in Mathematics}.
\newblock Springer-Verlag, New York, second edition, 1984.

\bibitem[Kow16]{Ko}
E.~Kowalski.
\newblock Arithmetic randonn\'{e}e. {A}n introduction to probabilistic number
  theory.
\newblock 2016.
\newblock Available at
  \url{https://people.math.ethz.ch/~kowalski/probabilistic-number-theory.pdf}.

\bibitem[KR14]{MR3248485}
Emmanuel Kowalski and Guillaume Ricotta.
\newblock Fourier coefficients of {$GL(N)$} automorphic forms in arithmetic
  progressions.
\newblock {\em Geom. Funct. Anal.}, 24(4):1229--1297, 2014.

\bibitem[KS16]{KoSa}
Emmanuel Kowalski and William~F. Sawin.
\newblock Kloosterman paths and the shape of exponential sums.
\newblock {\em Compos. Math.}, 152(7):1489--1516, 2016.

\bibitem[Leh76]{MR0429787}
D.~H. Lehmer.
\newblock Incomplete {G}auss sums.
\newblock {\em Mathematika}, 23(2):125--135, 1976.

\bibitem[Lox83]{MR737174}
J.~H. Loxton.
\newblock The graphs of exponential sums.
\newblock {\em Mathematika}, 30(2):153--163 (1984), 1983.

\bibitem[Lox85]{MR817102}
J.~H. Loxton.
\newblock The distribution of exponential sums.
\newblock {\em Mathematika}, 32(1):16--25, 1985.

\bibitem[Wey21]{Wey}
H.~Weyl.
\newblock Zur {A}bsch\"atzung von $\zeta(1+ti)$.
\newblock {\em Math. Z.}, 10:88--101, 1921.

\end{thebibliography}
\end{document}